\documentclass[oneside,reqno,a4paper,11pt]{amsart}
\usepackage{amsfonts}
\usepackage{amsmath,amsfonts,amssymb,amscd}
\usepackage{color,longtable,verbatim}
\usepackage{color}
\usepackage{latexsym, amsfonts,mathrsfs, amsmath, amssymb, amscd, epsfig}
\usepackage{mathrsfs}
\usepackage{amsthm}
\usepackage{etex}
\usepackage{epstopdf}
\usepackage{graphicx}
\usepackage{color}
\usepackage{ifpdf}
\usepackage{epstopdf}
\usepackage{yhmath}
\usepackage{booktabs}
\usepackage{multirow}
\usepackage{array}
\usepackage{subfigure}
\usepackage{esint}
\usepackage{float}
\usepackage{caption}%

\newtheorem {theorem}{Theorem}[section]
\newtheorem {lemma}[theorem]{{\bf Lemma}}

\newtheorem {proposition}[theorem]{{\bf Proposition}}
\theoremstyle{remark}

\theoremstyle{problem}

\theoremstyle{definition}
\newtheorem {definition}{{\bf Definition}}[section]
\theoremstyle{plain} \numberwithin {equation}{section}

\newtheorem{rmk}{Remark}[section]

\newtheorem{assumption}{Assumption}
\numberwithin{assumption}{section}

\begin{document}
	\vspace{1cm}

	\title[]{The free boundary of steady axisymmetric inviscid flow with vorticity $\uppercase\expandafter{\romannumeral1}$: near the degenerate point}
	
	\author[ Lili Du, \ \ Jinli Huang]{ Lili Du$^{\lowercase {1}}$, \ \ Jinli Huang$^{\lowercase {2}}$, \ \ Yang Pu$^{\lowercase {3}}$}
	\thanks{$^*$ This work is supported by National Nature Science Foundation of China Grant 11971331, 12125102, and Sichuan Youth Science and Technology Foundation 2021JDTD0024. }
\thanks{$^1$  E-mail: dulili@scu.edu.cn. $^2$ E-mail:
	huangjinli7253@163.com. $^3$ E-mail: puyang1011@126.com. Corresponding author}

\maketitle
\begin{center}
	Department of Mathematics, Sichuan University,
	
	Chengdu 610064, P. R. China.

\end{center}

\maketitle

\begin{abstract}
 In this paper, we investigate the singularity near the degenerate points of the  steady axisymmetric  flow with general vorticity  of an inviscid incompressible fluid   acted on by gravity and with a free surface. We called the points on the free boundary at which the gradient of the stream function vanishes as the degenerate points. The main results in this paper give the different classifications of the singularity near the degenerate points on the free surface. More precisely, we obtained that at the stagnation points, the possible profiles  must be  a Stokes corner, or a horizontal cusp, or a horizontal flatness. At the degenerate points on the symmetric axis except the origin, the wave profile must be a cusp. At the origin, the possible wave profiles must be  a Garabedian pointed bubble, or a horizontal cusp, or a horizontal flatness.
\end{abstract}

%
%
%
%


\section{Introduction and main results}

\subsection{Introduction}

In this paper and the subsequent paper \cite{DH}, we consider the singularity and the regularity of the free boundary of steady axisymmetric inviscid incompressible
flow with vorticity in gravity field near the degenerate points and the non-degenerate points, respectively. The steady axisymmetric incompressible ideal flow (the $y$-axis is the axis of symmetry) was governed by a semilinear nonhomogeneous elliptic equation with Bernoulli's type boundary condition
\begin{equation}\label{equation}
	\begin{cases}
		\begin{alignedat}{2}
			\operatorname{div}\left(\frac{1}{x}\nabla \psi\right)&=-x f(\psi) \quad && \text{ in } \Omega \cap\{\psi>0\},\\	\frac{1}{x^{2}}|\nabla \psi|^2&=-y  \quad && \text{ on }\Omega\cap\partial\{\psi>0\},
		\end{alignedat}
	\end{cases}
\end{equation}
where $\psi$ is the stream function and $\Omega$ is a connected open subset relative to the right half-plane $\mathbb{R}^{2}_{+}=\{(x,y)\in\mathbb{R}^{2}\mid x\geq 0)\}$. Here, $f(\psi)$ is the vorticity function  and is assumed to be continuous. And we denote $\partial\Omega$ the boundary of $\Omega$ relative to the right half-plane, that is, $\partial\Omega\cap\{x=0\}=\emptyset$. In the next section, we will give two physical models of the steady axisymmetric flow with general vorticity, which can explain the origin of the Bernoulli-type free boundary problem \eqref{equation}.

The dynamical boundary condition $|\nabla \psi|^2=\sqrt{-x^{2}y}$ on the free boundary $\Gamma:=\Omega\cap\partial\{\psi>0\}$ implies that the gradient of the stream function is degenerate at the symmetric axis and the $x$-axis, we called the free boundary point at which the gradient of the stream function vanishes as the \emph{degenerate point}. In particular, a free boundary point on the $x$-axis but not the origin is commonly known as a \emph{stagnation point}. In the present work, we consider the singularity of the solution near the degenerate points on the free boundary and the degenerate points are divided into three classes, namely,
\begin{equation*}
	\begin{aligned}
		&\text{Type }1. \quad\text{The stagnation point}\quad (\text{see Figure } \ref{picture:type12});\\
		&\text{Type }2. \quad\text{The degenerate point on the symmetric axis except the origin}\quad (\text{see Figure } \ref{picture:type12});\\
		&\text{Type }3. \quad\text{The  origin} \quad(\text{see Figure } \ref{picture:type3}).
	\end{aligned}
\end{equation*}

Similarly to the results on the steady axis-symmetric inviscid irrotational flow \cite{Varvaruca2012c}, the wave profile near the different types of the degenerate points possesses the different phenomena for the inviscid flow with vorticity.

\begin{figure}
	\begin{minipage}[t]{0.5\linewidth}
		\centering
		\includegraphics[width=2.2in]{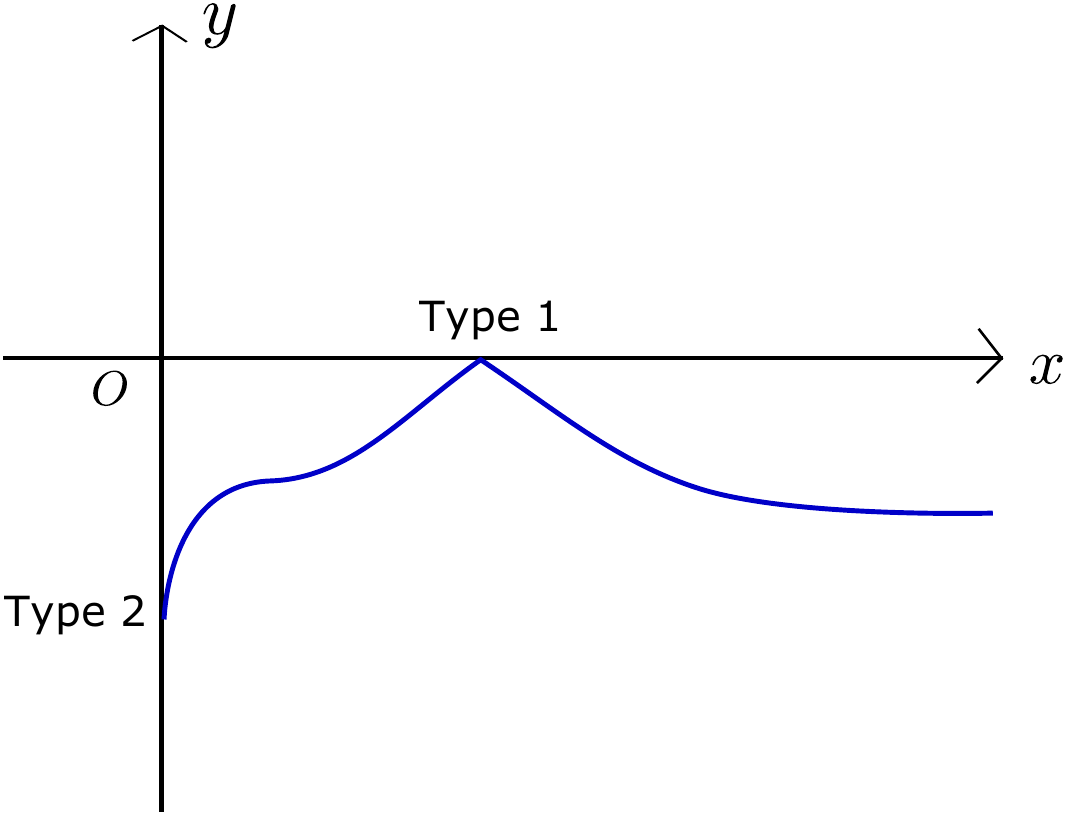}
		\caption{Type 1 and Type 2}
		\label{picture:type12}
	\end{minipage}%
	\begin{minipage}[t]{0.5\linewidth}
		\centering
		\includegraphics[width=2.2in]{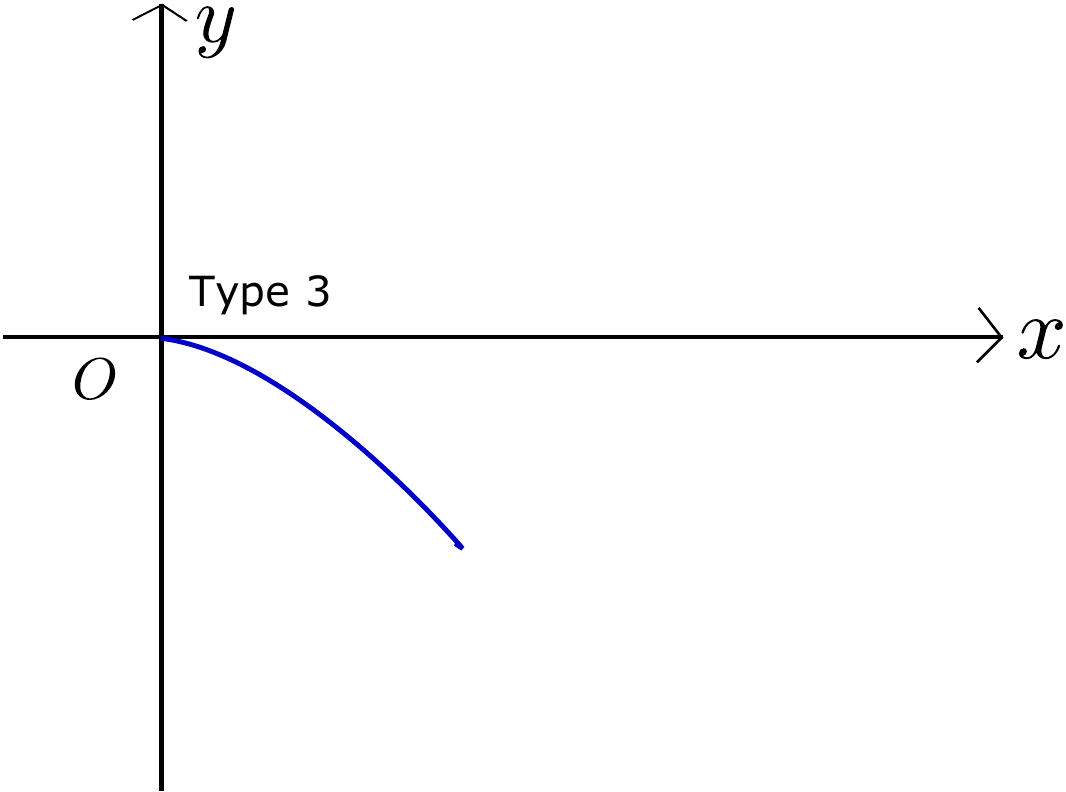}
		\caption{Type 3}
		\label{picture:type3}
	\end{minipage}
\end{figure}

 One of the behaviors of the axisymmetric solution near the Type 1 (the stagnation point), is similar to the one of two-dimensional case, which is related closely to the Stokes corner flow. This is known as the Stokes conjecture for the singularity of the two-dimensional irrotational steady water waves at the stagnation point. In 1880, G. Stokes conjectured that the free surface of two-dimensional irrotational water waves have a symmetric corner of $\frac{2\pi}{3}$ at the stagnation point \cite{Stokes1880} and that the free surface at the stagnation point is convex. There is a great body of research of
rigorous proofs of the Stokes conjecture. The earliest highlight in the research history was an article written by Nekrasov \cite{Nekrasov1921}, where he formulated the extreme-wave problem as an integral equation for a function $\theta(s)$ that gives the angle between the tangent to the free boundary and the horizontal line, namely,
 \begin{equation}\label{equation:nekrasov}
 	\theta_{\nu}(s)=\frac{1}{3\pi}\int_{0}^{\pi}\frac{\sin\theta_{\nu}(t)}{\nu^{-1}+\int_{0}{t}\sin\theta_{\nu}(u)\,du}\log\left\lvert\frac{\sin\frac{1}{2}(s+t)}{\sin\frac{1}{2}(s-t)}\right\rvert\,dt\text{, } s\in[0,\pi].
 \end{equation}
Here $\nu$ is a parameter depending on the period of the wave and the wave speed. In 1962, the first existence result for solutions of Nekrasov's integral equation  \eqref{equation:nekrasov} with $\nu>0$ was tackled by Krasovskii \cite{Krasovskii1962}, though his analysis was restricted to wave angles between $0$ and $\frac{\pi}{6}$.
Meanwhile, he conjectured that $\sup\limits_{s\in[0,\pi]}\theta_{\nu}(s)\leq\frac{\pi}{6}$ and there is no solution $(\nu,\theta_{\nu})$ with $\sup\limits_{s\in[0,\pi]}\theta_{\nu}(s)>\frac{\pi}{6}$. However, the refutation of the first part of the Krasovskii's conjecture was shown by Mcleod \cite{McLeod1997} for large $\nu>0$. In 1978, Krasovskii's results \cite{Krasovskii1962} were improved by Keady and Norbury \cite{Keady1978} using the modern global bifurcation theory. Toland \cite{Toland1978} then gave a proof that there was a convergence to a solution $\theta_{0}$ of the limiting problem $\nu=\infty$ and he concluded that the limit $\lim\limits_{s\rightarrow0+}\theta_{0}(s)=\frac{\pi}{6}$, provided that it existed. Until 1982, Amick, Fraenkel, and Toland \cite{Amick1982} and Plotnikov \cite{Plotnikov2002}, with proofs dependent on conformal mappings, proved the Stokes conjecture independently for isolated singularities satisfying some structural assumptions on the isolatedness of stagnation point, and the symmetry and monotonicity of free surface. Subsequently, the convexity of Stokes waves of extreme form and the second part of the Krasovskii's conjecture were confirmed by Plotnikov and Toland \cite{Plotnikov}. As an important breakthrough in the analysis of the Stokes conjecture, in 2011, Varvaruca and Weiss \cite{Varvaruca2011} applied a new geometric approach, mainly based on the blow-up limits, the monotonicity formula, and the frequency formula, to give a significant proof of Stokes conjecture without any structural assumptions.

Although Stokes seems to have made the conjecture only for irrotational waves, Varvaruca \cite{Varvaruca2009} first studied the existence of extreme waves with vorticity and their properties. Before Varvaruca \cite{Varvaruca2009}, the existence of steady, periodic, small-amplitude water waves with general vorticity had been considered in Dubreil-Jacotin \cite{Dubreil}. While Constantin and Strauss \cite{Constantin} also constructed rotational waves of large amplitude, they were unable to imply whether there was a limiting wave having a stagnation point. Varvaruca's paper \cite{Varvaruca2009} was the first pioneering breakthrough in the existence of extreme waves with vorticity, he proved that a symmetric monotone free boundary at stagnation points has either a corner of $\frac{2\pi}{3}$ or a horizontal tangent. Furthermore, he showed that if the vorticity function is non-negative close to the free surface, then the free surface necessarily has a corner of $\frac{2\pi}{3}$. Special mention should be made of the important contributions of the results of Varvaruca and Weiss \cite{Varvaruca2012a}, that is, an interesting new feature of the wave with vorticity at the stagnation point could probably exist. Significantly, the new feature is called cusp, which would not be possible without the presence of vorticity. Conversely, if either the vorticity vanishes on the free surface or the vorticity function is non-negative, and the free surface is an injective curve, horizontally flat singularities can be excluded.

This new geometric approach in \cite{Varvaruca2011} and \cite{Varvaruca2012a} can still be applied to analysis of the singularity of the free surface to the three-dimensional axially symmetric water wave problem. As a first result for three-dimensional case, Varvaruca and Weiss \cite{Varvaruca2012c} analyzed the possible profiles of the water wave and the free boundaries close to stagnation points, as well as points on the axis of symmetry of the axisymmetric, three-dimensional, inviscid, incompressible fluid acted on by gravity and with a free surface. To the best of our knowledge, Varvaruca and Weiss \cite{Varvaruca2012c} was the earliest article to rigorously analyze the singularity of the free boundary near the degenerate points on the symmetric axis. Most notably they found that one of the possible profiles at the original point is the Garabedian pointed bubble, which is a new feature in the axially symmetric problem and was first found by Garabedian \cite{Garabedian1985}. Besides, Garcia, Varvaruca and
Weiss \cite{Smit2016} first use this geometric method to the ElectroHydroDynamic equations (EHD). Also, they found a unique homogeneous solution of the EHD equations, which
seems to correspond to the well-known "Garabedian pointed bubble" in fluid flow
without electric field.

 On another hand, there are numerous papers concerning the regularity away from the degenerate points.
 As a pioneer work in this field, Alt, Caffarelli and Friedman proved in \cite{Alt1981} the regularity of the free boundary using a variational approach. In particular, they obtained the singularities can not occur in two dimensions.
In \cite{Alt1982}, the authors had shown the regularity away from the degenerate points in a three-dimensional axially symmetric cavity without gravity. Moreover, it was shown that the free boundary is $C^{1,\alpha}$ under the assumption that the vorticity function $f(\psi)$ is a constant in \cite{Friedman1983}. Recently, the mathematical results of \cite{Alt1981} were used
 in the study of compressible impinging jet flows \cite{CDW}, two phase fluids \cite{CDX} and fluid issuing from de Laval nozzle \cite{CD}.

The study in this paper aims to generalize the results to the axisymmetric flow with vorticity, with the novel use of a nonhomogeneous semilinear elliptic equation. A very useful tool in the present work is the monotonicity formula, which is an extension of the monotonicity formulas in Theorem 3.1 of  \cite{Weiss1998} and in Theorem 1.2 of \cite{Weiss19992}.  Besides, the Weiss's monotonicity formula (see also Theorem 2 in \cite{Weiss1999})  plays an important role in studying free boundary problems of a variational nature with a homogeneous structure. The major objectives of the present paper are to analyze all the possible blow-up limits of the free boundary close to different types of the degenerate points.
Due to the appearance of general vorticity, all of the possible profiles could exist without any further restriction.

\subsection{Definitions and notations}
In this subsection, we will give some definitions on the solutions of the
free boundary problem \eqref{equation}, and some notations in this paper.

First, we introduce the definition of the weighted Sobolev space $W_{w,loc}^{1,2}(\Omega)$, which is useful for us to  define the solutions of the problem \eqref{equation} in some suitable senses. For simplicity of notation, we take $X=(x,y)\in\mathbb{R}^{2}$ and $dX=dxdy$ here and afterwards.

\begin{definition}Let $E\subset\mathbb{R}^{2}_{+}$ be an open set.
The weighted space $L^{2}_{w}(E)$  and the local space $L^{2}_{w,loc}(E)$ are defined as
	\begin{equation*}
	L^{2}_{w}(E):=\left\{g:E\rightarrow \mathbb{R};g \text{ is measurable and }\int_{E}\frac{1}{x}|g|^{2}\,dX<+\infty\right \}
	\end{equation*}
	and \begin{equation*}
	L^{2}_{w,loc}(E):=\{g\in L^{2}_{w}(K)  \text{ for any compact set }K \text{ of }E\}.
	\end{equation*}
	The norm of the space $L^{2}_{w}(E)$  is defined by
	\begin{equation*}
	||g||_{L^{2}_{w}(E)}=\bigg(\int_{E}\frac{1}{x}|g|^{2}\,dX\bigg)^{1/2}.
	\end{equation*}

 Similarly, the weighted Sobolev space $W^{1,2}_{w}(E)$ and the local weighted Sobolev space $W^{1,2}_{w,loc}(E)$ are defined to be
	\begin{equation*}
	W^{1,2}_{w}(E):=\left\{g\in L^{2}_{w}(E); \frac{\partial g}{\partial x}\in L^{2}_{w}(E)\text{ and }\frac{\partial g}{\partial y}\in L^{2}_{w}(E)\right\},
	\end{equation*}
	and
	\begin{equation*}
	W^{1,2}_{w,loc}(E):=\left\{g\in L^{2}_{w,loc}(E); \frac{\partial g}{\partial x}\in L^{2}_{w,loc}(E)\text{ and }\frac{\partial g}{\partial y}\in L^{2}_{w,loc}(E)\right\},
	\end{equation*}	
	where $\frac{\partial }{\partial x}$ and  $\frac{\partial }{\partial y}$ are the first-order weak partial derivatives.
\end{definition}

Next, we would like to give the definition of a variational solution to the problem \eqref{equation}.
Note that, the vorticity function $f$ is assumed to be a continuous function throughout this paper.

\begin{definition}
	A function $\psi\in W_{w,loc}^{1,2}(\Omega)$ is called a variational solution of the problem \eqref{equation} provided

	(\romannumeral1) $\psi\in C^{0}(\Omega)\cap C^{2}(\Omega\cap\{\psi>0\}),\psi\geq 0$ in $\Omega$ and $\psi=0$ on $\{x=0\}\cap\Omega$,
	
	(\romannumeral2) $\lim\limits _{X \rightarrow X_{0}, \atop X \in \Omega \cap\{\psi>0\}}\frac{1}{x} \frac{\partial \psi}{\partial y}=0\quad$  and $\lim\limits _{X \rightarrow X_{0}, \atop X \in \Omega \cap\{\psi>0\}}\frac{1}{x} \frac{\partial \psi}{\partial x}$ exist \\for any $X_{0}\in\Omega\cap\{x=0\}$,
	
	(\romannumeral3) the first variation with respect to domain variations of the functional
	\begin{equation*}
	J(\tilde{\psi})=\int_{\Omega}\left(\frac{1}{x}|\nabla \tilde{\psi}|^{2}-2xF(\tilde{\psi})-xyI_{\{\tilde{\psi}>0\}}\right)\,dX
	\end{equation*}
	vanishes at $\tilde{\psi}=\psi$, where $F(t):=\int_{0}^{t}f(s)ds$, and $I_{E}$ is the characteristic function of a set $E$.
	Equivalently,
	\begin{align}\label{equation:variationalSolution}
	0&=-\frac{d}{d\varepsilon}J\left(\psi\left(X+\varepsilon\eta(X)\right)\right)\rvert_{\varepsilon=0}\notag\\
	&=\int_{\Omega}\bigg(\frac{1}{x}|\nabla\psi|^{2}\nabla\cdot\eta-\frac{1}{x^{2}}|\nabla \psi|^{2}\eta_{1}-\frac{2}{x}\nabla\psi D\eta\cdot\nabla\psi\notag\\
	&~~~~- \operatorname{div}(xy\eta)I_{\{\psi>0\}}-2F(\psi)\operatorname{div}(x\eta)\bigg)\,dX
	\end{align}
	for each $\eta(X)=(\eta_{1}(X),\eta_{2}(X))\in C^{1}_{0}(\Omega;\mathbb{R}^{2})$ satisfying $\eta_{1}=0$ on $\{x=0\}\cap\Omega$,
	
	(\romannumeral4) the free boundary $\partial\{\psi>0\}$ is contained in the quarter plane $\{(x,y)\mid x\geq0, y\leq 0\}$.
\end{definition}

\begin{rmk}
 It follows that the variational solution $\psi$ of the problem \eqref{equation} satisfies $\operatorname{div}\left(\frac{1}{x}\nabla \psi\right)=-xf(\psi)$ in $\Omega\cap\{\psi>0\}$ in  sense  of distributions and $\frac{1}{x^{2}}|\nabla\psi|^{2}=-y$ on $C^{2,\alpha}$-smooth parts of the free boundary $\partial\{\psi>0\}\cap\{xy\neq0\}$ according to an integration by parts.
\end{rmk}

Also, we give the definition of a weak solution of the problem \eqref{equation}.

\begin{definition}
	A function $\psi\in W^{1,2}_{w,loc}(\Omega)$ is called a weak solution of the problem \eqref{equation} provided
	
	(\romannumeral1) $\psi$ is a variational solution of the problem \eqref{equation},\\and
	
	(\romannumeral2)
	the free boundary $\partial\{\psi>0\}\cap\Omega\cap\{x>0\}\cap\{y\neq0\}$ is locally a $C^{2,\alpha}$-smooth curve.
\end{definition}

Since our results are completely local, we only consider a neighborhood of the degenerate point $X_{0} \in \Omega \cap \partial\{\psi>0\}$, and therefore we take the notations as follows
\begin{equation*}
	B_{r}(X_{0}):=\{X=(x,y)\in\mathbb{R}^{2}\mid|X-X_{0}|<r\},
\end{equation*}
and
\begin{equation*}
B_{r}^{+}(X_{0}):=\{X=(x,y)\in\mathbb{R}^{2}\mid|X-X_{0}|<r\text{ and }x>0\}.
\end{equation*}
For simplicity, we take $O:=(0,0)$, $B_r:=B_r(O)$ and $B_{r}^{+}:=B_{r}^{+}(O)$ throughout this paper.

Moreover, we denote the sets of the degenerate point Type 1 and Type 2 by
\begin{equation*}
	S^{s}_{\psi}:=\{X_{0}=(x_{0},0)\in\Omega\cap\partial\{\psi>0\}; x_{0}> 0\}
\end{equation*}  and
\begin{equation*}
	S^{a}_{\psi}:=\{X_{0}=(0,y_{0})\in\Omega\cap\partial\{\psi>0\}; y_{0}< 0\}
\end{equation*}  respectively.

Finally, we define the blow-up sequence for three classes degenerate points. It should be noted that the gradient condition $\lvert\nabla\psi\rvert=\sqrt{-x^{2}y}$ on the free boundary implies the behavior of the solution near the degenerate point. More precisely, for Type 1 degenerate point $X_{0}=(x_{0},0)\, (x_{0}>0)$, $\psi$ goes like $r^{3/2}$ near $X_{0}$; For Type 2 degenerate point $X_{0}=(0,y_{0})\, (y_{0}<0)$, $\psi$ goes like $r^{2}$ near $X_{0}$; For Type 3 degenerate point $X_{0}=(0,0)$, $\psi$ goes like $r^{5/2}$ near $X_{0}$. These facts suggest us to define the following three types of the blow-up subsequence.

\begin{definition}
Let $r_{m}>0$ converging to $0$ as $m\rightarrow+\infty$, the blow-up sequence $\{\psi_{m}(X)\}$ is defined by
\begin{flalign}
    && \text{ (Type 1.)\ \ }
 \psi_{m}(X)&:=
\frac{\psi(X_{0}+r_{m}X)}{r_{m}^{3/2}} \text{ for any }X_{0}\in S_{\psi}^{s};  & \label{equation:blowUp1} \\
    && \text{ (Type 2.)\ \ }
\psi_{m}(X)&:=\frac{\psi(X_{0}+r_{m}X)}{r_{m}^{2}} \text{ for any }X_{0}\in S_{\psi}^{a}; & \label{equation:blowUp2} \\
    && \text{ (Type 3.)\ \ }
\psi_{m}(X)&:=\frac{\psi(X_{0}+r_{m}X)}{r_{m}^{5/2}}\text{ for }X_{0}=O, & \label{equation:blowUp3}
\end{flalign}
which is well-defined for $|X|<\frac{1}{r_m}dist(X_0, \partial\Omega)$, respectively.
Notice that for every $r>0$, if the functions $\psi_{m}(X)$ are uniformly bounded in $W^{1,2}_{loc}(B_r)$ $(W^{1,2}_{w,loc}(B_r^+))$ for $X_0\in S_{\psi}^{d}$ $(X_0\in S_{\psi}^{a}\cup\{O\})$, thus there exists a (not-relabeled) subsequence $\{\psi_{m}(X)\}$ and a function $\psi_{0}(X)\in W^{1,2}_{loc}(\mathbb{R}^2)$ $(W^{1,2}_{loc}(\mathbb{R}_{+}^2))$, such that $\psi_{m}(X)$ converges to $\psi_{0}(X)$ weakly in $W_{loc}^{1,2}(\mathbb{R}^{2})$ $(W^{1,2}_{loc}(\mathbb{R}_{+}^2))$ as $m\rightarrow+\infty$. Such a function $\psi_{0}$ is called a \emph{blow-up limit} of $\psi$ at $X_{0}$, which give the infinitesimal behavior of $\psi$ near $X_{0}$. Since different subsequences may converge to different blow-up limits,  one of the main goals in this paper is to compute all possible blow-up limits for Type 1, Type 2 and Type 3.
\end{definition}

\subsection{Main results}

Before we state the main results, we would like to give two important assumptions.
\begin{assumption}\label{assumption1}
	For any $X_{0}=(x_{0},y_{0})\in \Omega\cap\partial\{\psi>0\}$, suppose
	\begin{equation}\label{equation:a1}
	\frac{|\nabla\psi(x,y)|^{2}}{x^{2}}\leq C(|y|+|x-x_{0}|)\quad\text{  in }B^{+}_{r_{0}}(X_{0}),
	\end{equation}
	where $r_{0}>0$ is sufficiently small, and $C$ is a positive constant.
	
	In addition, \eqref{equation:a1} implies that
	\begin{equation*}
		\psi^{2}(x,y)\leq Cx^{2}(|y|+|x-x_{0}|)\lvert X-X_{0}\rvert^{2}\quad\text{  in }B^{+}_{r_{0}}(X_{0}).
	\end{equation*}
\end{assumption}

\begin{assumption}\label{assumption:2}
	Supposing that  in a neighborhood of $X_{0}\in \Omega \cap \partial\{\psi>0\}$, there exists $r_{0}>0$ such that $\partial\{\psi>0\}\cap B_{r_{0}}^{+}(X_{0})$ is a continuous injective curve, which can be written by $g(t)=(g_{1}(t),g_{2}(t)):I\rightarrow\mathbb{R}^{2}$ such that $g(0)=X_{0}$, where $I$ is an interval of $\mathbb{R}$ containing the origin.
\end{assumption}

In order to compute the blow-up limits, for any $X_{0}\in \Omega \cap \partial\{\psi>0\}$, we need to define the following functionals.
\begin{align*}
\mathscr{D}_{1,X_{0},\psi}(r)&=\int_{B_{r}^{+}(X_{0})}\left(\frac{1}{x}|\nabla \psi|^{2}-xyI_{\{\psi>0\}}-x\psi f(\psi)\right)\,dX,\\
\mathscr{D}_{2,X_{0},\psi}(r)&=\int_{\partial B_{r}^{+}(X_{0})}\frac{1}{x}\psi^{2}\,d\mathcal{H}^{1},\\
\mathscr{D}^{y}_{X_{0},\psi}(r)&=r^{-3}\mathscr{D}_{1,X_{0},\psi}(r)-\frac{3}{2}r^{-4}\mathscr{D}_{2,X_{0},\psi}(r),\\
\mathscr{D}^{x}_{X_{0},\psi}(r)&=r^{-3}\mathscr{D}_{1,X_{0},\psi}(r)-2r^{-4}\mathscr{D}_{2,X_{0},\psi}(r),\\
\text{and}\ \ \ \ \ \ \ \ \ \ \ \
\mathscr{D}^{xy}_{X_{0},\psi}(r)&=r^{-4}\mathscr{D}_{1,X_{0},\psi}(r)-\frac{5}{2}r^{-5}\mathscr{D}_{2,X_{0},\psi}(r),
\end{align*}
where $\mathcal{H}^{1}$ is 1-dimensional Hausdorff measure.
\begin{rmk}
In the following discussion, we will find that
the limit $\mathscr{D}^{y}_{X_{0},\psi}(r)$, $\mathscr{D}^{x}_{X_{0},\psi}(r)$ and $\mathscr{D}^{xy}_{X_{0},\psi}(r)$ for Type 1, Type 2 and Type 3 exist separately as $r\rightarrow 0+$.
And we denote the limits by $\mathscr{D}^{y}_{X_{0},\psi}(0+)$,
 $\mathscr{D}^{x}_{X_{0},\psi}(0+)$ and $\mathscr{D}^{xy}_{X_{0},\psi}(0+)$, respectively. Moreover, notice that
$$ \mathrm{ (Type~~1.)\ \ for \ \ }\psi_{m}(X):=
\frac{\psi(X_{0}+r_{m}X)}{r_{m}^{3/2}}\mathrm{\ \ and \ \ }X_{0}\in S_{\psi}^{d},$$
$$\mathscr{D}^{y}_{X_{0},\psi}(0+)=-x_0\lim_{m\rightarrow +\infty}\int_{B_1}yI_{\{\psi_m>0\}}\,dX=-x_0\lim_{m\rightarrow +\infty}\frac{w_2\int_{B_1}yI_{\{\psi_{m}>0\}}\,dX}{|B_1|}$$
$$  \mathrm{ (Type~~2.)\ \ for \ \ }\psi_{m}(X):=
\frac{\psi(X_{0}+r_{m}X)}{r_{m}^{2}}\mathrm{\ \ and \ \ }X_{0}\in S_{\psi}^{a},$$
$$\mathscr{D}^{x}_{X_{0},\psi}(0+)=-y_0\lim_{m\rightarrow +\infty}\int_{B_1^{+}}xI_{\{\psi_m>0\}}\,dX=-y_0\lim_{m\rightarrow +\infty}\frac{w_2\int_{B_1^+}xI_{\{\psi_m>0\}}\,dX}{2|B_1^+|}$$
$$  \mathrm{ (Type~~3.)\ \ for \ \ } 
\psi_{m}(X):=
\frac{\psi(X_{0}+r_{m}X)}{r_{m}^{5/2}}\mathrm{\ \ and \ \ }X_{0}=O,$$
$$\mathscr{D}^{xy}_{X_{0},\psi}(0+)=-\lim_{m\rightarrow +\infty}\int_{B_1^+}xyI_{\{\psi_m>0\}}\,dX=-\lim_{m\rightarrow +\infty}\frac{w_2\int_{B_1^+}xyI_{\{\psi_m>0\}}\,dX}{2|B_1^+|},$$
where $w_2$ is the volume of the unit ball in two dimensions, which are equivalent to the weighted density of the set $\{\psi>0\}$ at $X_{0}$ in the case Type 1, Type 2 and Type 3 respectively in some suitable sense.
\end{rmk}

The following are our main results. Note that we will give all the possible profiles of the free boundaries near three types of the degenerate points, please see Table 1.
\begin{table}[H]
\caption{Blow-up limits}
\centering
\resizebox{.95\columnwidth}{!}{
\renewcommand{\arraystretch}{1.5}
\begin{tabular}{|m{0.7cm}<{\centering}|m{7cm}<{\centering}|m{5.7cm}<{\centering}|m{6cm}<{\centering}|}
\hline
Type & Blow-up limits & Weighted density & Wave profile\\
\hline
\multirow{5}{*}{1} & $\psi_0(x,y)=\psi_{0}(r \sin\theta, r \cos\theta)=$ & $-x_0\int_{B_1}yI_{\{(r sin\theta,r cos\theta)|\frac{2}{3}\pi<\theta<\frac{4}{3}\pi\}}dX$ & Stoke corner,
see Fig. 1
\\
& $\frac{\sqrt{2}}{3}x_0 r^{\frac{3}{2}}cos\left(\frac{3}{2}\theta-\frac{\pi}{2}\right)I_{\{(r sin\theta, r cos\theta)|\frac{2}{3} \pi<\theta<\frac{4}{3}\pi\}}$& &
\\
\cline{2-4}
& \multirow{3}{*}{$\psi_0(x,y)\equiv0$}& $-x_0\int_{B_1}y^{+}dX$ &
Horizontally flatness, see Fig. 2 (a)
\\
\cline{3-4}
& & $-x_0\int_{B_1}y^{-}dX$ &
Horizontally flatness, see Fig. 2 (b)
\\
\cline{3-4}
& & $0$ &
Horizontally flatness or cusp, see Fig. 3-4
\\
\hline
\multirow{3}{*}{2} & $\psi_{0}(x, y)=Cx^2$ & $-y_0\int_{B_1^{+}}xdX$ &
One of three vertical cusps, see Fig. 5-6
\\
\cline{2-4}
& \multirow{2}{*}{$\psi_0(x,y)\equiv0$}& $-y_0\int_{B_1^{+}}xdX$ &
One of three vertical cusps, see Fig. 5-6
\\
\cline{3-4}
& & $0$ &
Cusp, see Fig. 7
\\
\hline
\multirow{5}{*}{3} & $\psi_0(x,y)=\psi_{0}(r sin\theta, r cos\theta)=$ & $-\int_{B^{+}_{1}}xyI_{\{(rsin\theta,r cos\theta)|0<\theta<\theta^*\}}dX$ &
Garabebian pointed bubble, see Fig. 8
\\
& $C_0 r^{\frac{5}{2}}sin^2\theta P'_{3/2}(cos\theta)I_{\{(r sin\theta, r cos\theta)|0<\theta<\theta^*\}}$
& &
\\
\cline{2-4}
& \multirow{3}{*}{$\psi_0(x,y)\equiv0$}& $-\int_{B^{+}_{1}}xy^{+}dX$ &
Horizontally flatness, see Fig. 9 (a)
\\
\cline{3-4}
& & $-\int_{B^{+}_{1}}xy^{-}dX$ &
Horizontally flatness, see Fig. 9 (b)
\\
\cline{3-4}
& & $0$ &
Horizontally cusp, see Fig. 10
\\
\hline
\end{tabular}
}
\end{table}

The first result establishes the possible profiles of the free boundaries close to Type 1 degenerate point.

\begin{theorem}[Type 1 degenerate point]\label{theorem:1}
	Let $\psi$ be a weak solution of the problem \eqref{equation} satisfying Assumption \ref{assumption1} and Assumption \ref{assumption:2}. Then for any $X_{0}\in S^{s}_{\psi}$,
	\begin{equation*}
		\begin{aligned}
			\mathscr{D}^{y}_{X_{0},\psi}(0+)\in\Bigg\{&-x_{0}\int_{B_{1}\cap\left\{(r\sin\theta,r\cos\theta);\frac{2\pi}{3}<\theta<\frac{4\pi}{3}\right\}}y\,dX,\\&-x_{0}\int_{B_{1}}y^{+}\,dX,-x_{0}\int_{B_{1}}y^{-}\,dX,0\Bigg\}.
		\end{aligned}
	\end{equation*}
Here and afterwards, $y^{+}=\max(0,y)$ and $y^{-}=\min(0,y)$. Furthermore, due to the classification of weighted density, there are exactly three cases that could happen.
	
	\textbf{Case 1. Stokes corner.} If $\mathscr{D}^{y}_{X_{0},\psi}(0+)=-x_{0}\int_{B_{1}\cap\{(r\sin\theta,r\cos\theta);\frac{2\pi}{3}<\theta<\frac{4\pi}{3}\}}y\,dX$,  then either \begin{equation*}
	\lim\limits_{t\rightarrow0+}\frac{g_{2}(t)}{g_{1}(t)-x_{0}}=\frac{1}{\sqrt{3}} \quad\text{ and }\quad \lim\limits_{t\rightarrow0-}\frac{g_{2}(t)}{g_{1}(t)-x_{0}}=-\frac{1}{\sqrt{3}}
	\end{equation*}
	or\begin{equation*}
	\lim\limits_{t\rightarrow0+}\frac{g_{2}(t)}{g_{1}(t)-x_{0}}=-\frac{1}{\sqrt{3}} \quad\text{ and }\quad \lim\limits_{t\rightarrow0-}\frac{g_{2}(t)}{g_{1}(t)-x_{0}}=\frac{1}{\sqrt{3}}
	\end{equation*}
	and $g_{1}(t)\neq x_{0}$ in $t\in(-t_{1},t_{1})\backslash\left\{0\right\}$, where $t_{1}$ is a positive constant, see Fig. 1.
\begin{figure}[H]
		\centering
		\includegraphics[width=6.4cm, height=4.8cm]{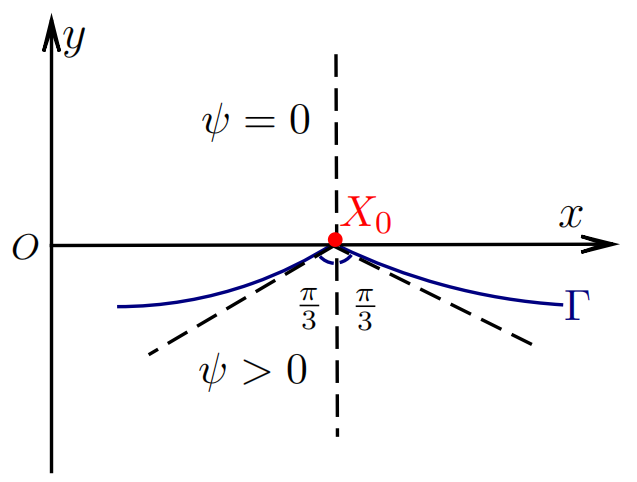}
		\caption*{Fig. 1. Stokes corner}
\end{figure}
\textbf{Case 2. Horizontally flatness.} If $\mathscr{D}^{y}_{X_{0},\psi}(0+)\in\left\{-x_{0}\int_{B_{1}}y^{+}\,dX,-x_{0}\int_{B_{1}}y^{-}\,dX\right\}$,  then \begin{equation*}
	\lim\limits_{t\rightarrow0}\frac{g_{2}(t)}{g_{1}(t)-x_{0}}=0,
	\end{equation*}
	$g_{1}(t)\neq x_{0}$ in $t\in(-t_{1},t_{1})\backslash\left\{0\right\}$, and
 $g_{1}(t)-x_{0}$ changes sign at $t=0$, see Fig. 2.
\begin{figure}[H]
	\begin{minipage}[t]{0.48\linewidth}
		\centering
		\includegraphics[width=6.4cm, height=4.8cm]{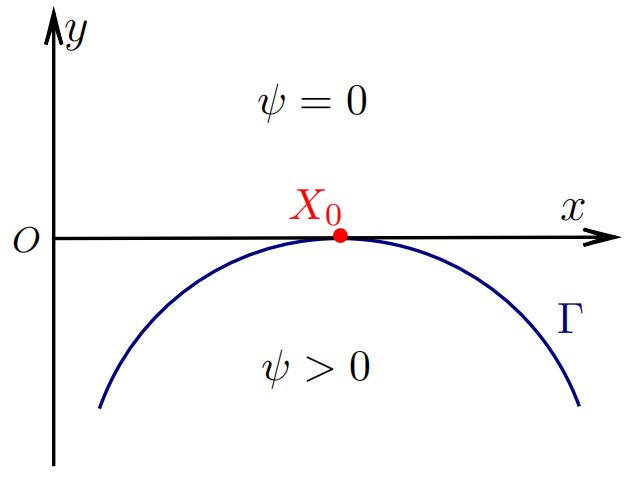}
\caption*{(a)}
	\end{minipage}%
	\begin{minipage}[t]{0.48\linewidth}
		\centering
		\includegraphics[width=6.4cm, height=4.8cm]{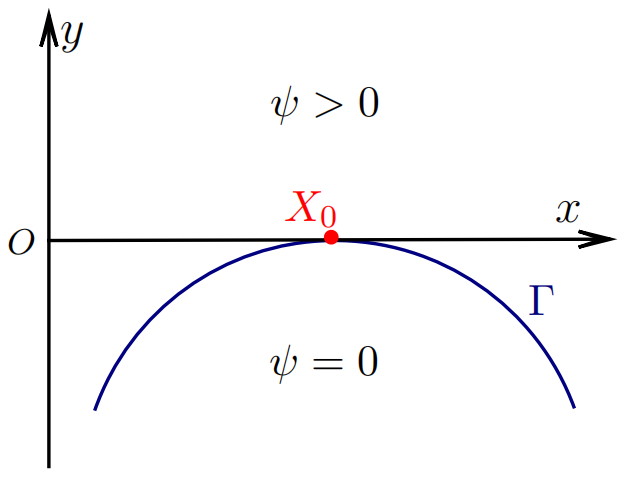}
\caption*{(b)}
	\end{minipage}%
		\caption*{Fig. 2. Horizontally flatness}
\end{figure}
\textbf{Case 3.} If $\mathscr{D}^{y}_{X_{0},\psi}(0+)=0$, then
	\begin{equation*}
\lim\limits_{t\rightarrow0}\frac{g_{2}(t)}{g_{1}(t)-x_{0}}=0,
	\end{equation*}
	$g_{1}(t)\neq x_{0}$ in $t\in(-t_{1},t_{1})\backslash\{0\}$, and there are two different cases.

\textbf{Subcase 3.1. Horizontally flatness.}
$g_{1}(t)-x_{0}$ changes sign at $t=0$, see Fig. 3.
\begin{figure}[H]
		\includegraphics[width=6.4cm, height=4.8cm]{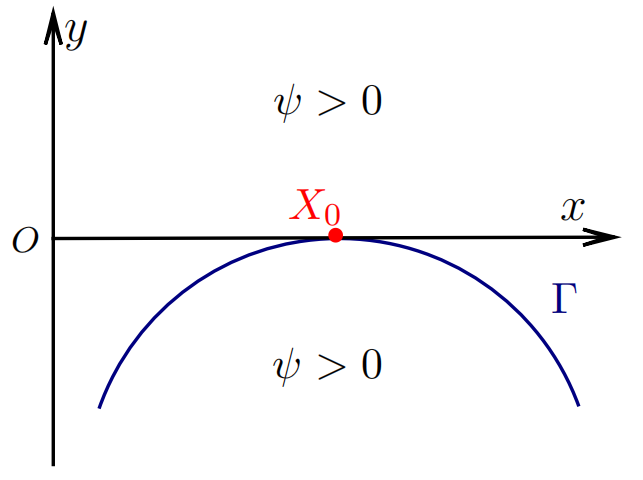}
\caption*{Fig. 3. Horizontally flatness}
\end{figure}

\textbf{Subcase 3.2. Horizontally cusp.}
$g_{1}(t)-x_{0}$ does not change sign at $t=0$, see Fig. 4.
\begin{figure}[H]
	\begin{minipage}[t]{0.48\linewidth}
		\centering
		\includegraphics[width=6.4cm, height=4.8cm]{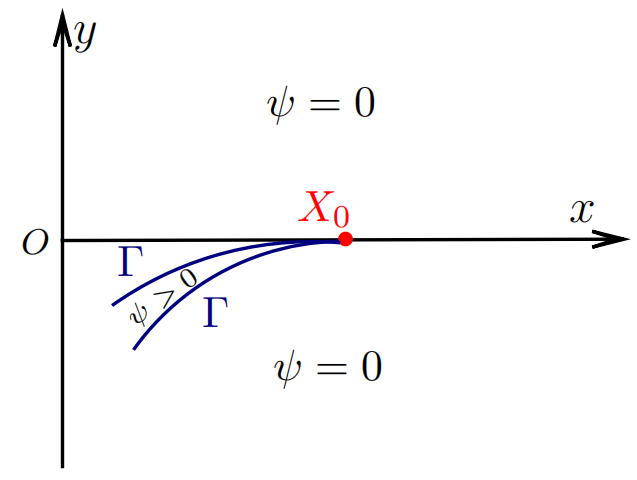}
	\end{minipage}%
	\begin{minipage}[t]{0.48\linewidth}
		\centering
		\includegraphics[width=6.4cm, height=4.8cm]{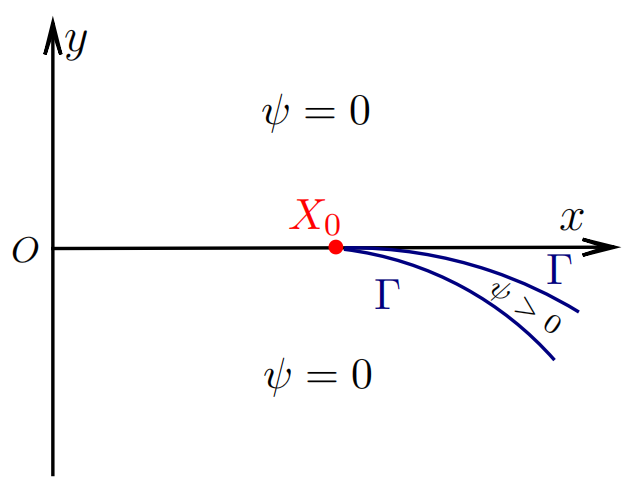}
	\end{minipage}%
		\caption*{Fig. 4. Horizontally cusp}
\end{figure}
Furthermore, the blow-up limits
$$\psi_{0}(x,y)=\psi_{0}(r sin\theta, r cos\theta)=\frac{\sqrt{2}}{3}x_0 r^{\frac{3}{2}}cos\left(\frac{3}{2}\theta-\frac{\pi}{2}\right)I_{\{(r \sin\theta, r \cos\theta):\frac{2}{3} \pi<\theta<\frac{4}{3}\pi\}}\text{ for Case 1 }$$
and
$$\psi_{0}(x,y)\equiv0\text{ for Case 2 and Case 3}.$$
\end{theorem}

The second result deals with the possible profiles of the free boundaries close to Type 2 degenerate point.

\begin{theorem}[Type 2 degenerate point]\label{theorem:2}
	 Let  $\psi$ be a weak solution of the problem \eqref{equation} satisfying Assumption \ref{assumption1} and Assumption \ref{assumption:2}. Then for any $X_{0}\in S^{a}_{\psi}$,
	 $$\mathscr{D}^{x}_{X_{0},\psi}(0+)\in \Bigg\{-y_{0}\int_{B_{1}^{+}}x\,dX,0\Bigg\}.$$
	  Moreover, there are exactly two cases that could happen due to the classification of density.

	\textbf{Case 1. Vertical cusp.} If $\mathscr{D}^{x}_{X_{0},\psi}(0+)=-y_{0}\int_{B_{1}^{+}}x\,dX$, then
either $g_{2}(t)\neq y_{0}$ in $t\in(0,t_{1})$ and \begin{equation*}
	\lim\limits_{t\rightarrow0+}\frac{g_{1}(t)}{g_{2}(t)-y_{0}}=0 \text{ (see Fig. 5)},
	\end{equation*}
\begin{figure}[H]
	\begin{minipage}[t]{0.48\linewidth}
		\centering
		\includegraphics[width=3cm, height=5cm]{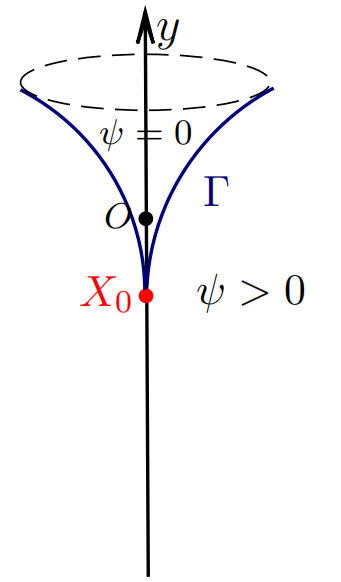}
\caption*{(a) Up vertical cusp}
	\end{minipage}%
	\begin{minipage}[t]{0.48\linewidth}
		\centering
		\includegraphics[width=3cm, height=5cm]{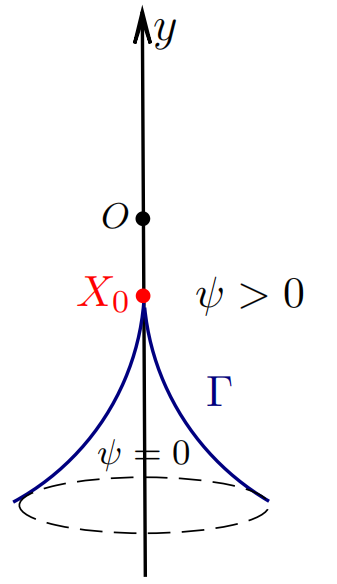}
\caption*{(b) Down vertical cusp}
	\end{minipage}%
		\caption*{Fig. 5. Vertical cusp}
\end{figure}
	or $g_{2}(t)\neq y_{0}$ in $t\in(-t_{1},t_{1})\backslash\{0\}$, $g_{2}(t)-y_{0}$ changes sign at $t=0$, and  \begin{equation*}
	\lim\limits_{t\rightarrow0}\frac{g_{1}(t)}{g_{2}(t)-y_{0}}=0 \text{ (see Fig. 6)}.
	\end{equation*}
\begin{figure}[H]
		\includegraphics[width=3cm, height=5cm]{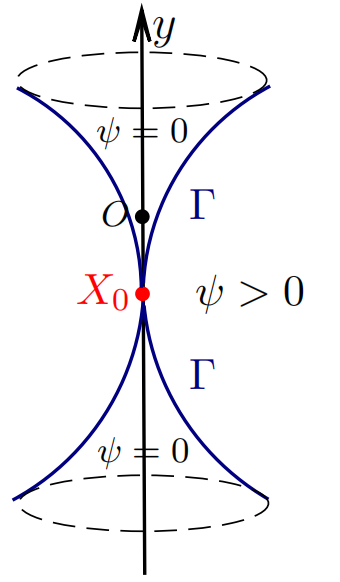}
\caption*{Fig. 6. Double vertical cusp}
	\end{figure}
\textbf{Case 2. Cusp.} If $\mathscr{D}^{x}_{X_{0},\psi}(0+)=0$, then
 $g_{2}(t)\neq y_{0}$ in $t\in(-t_{1},t_{1})\backslash\{0\}$, $g_{2}(t)-y_{0}$ does not change sign at $t=0$, and  \begin{equation*}
	\lim\limits_{t\rightarrow0}\frac{g_{1}(t)}{g_{2}(t)-y_{0}}=0 \text{ (see Fig. 7)}.
	\end{equation*}
\begin{figure}[H]
	\begin{minipage}[t]{0.48\linewidth}
		\centering
		\includegraphics[width=3cm, height=5cm]{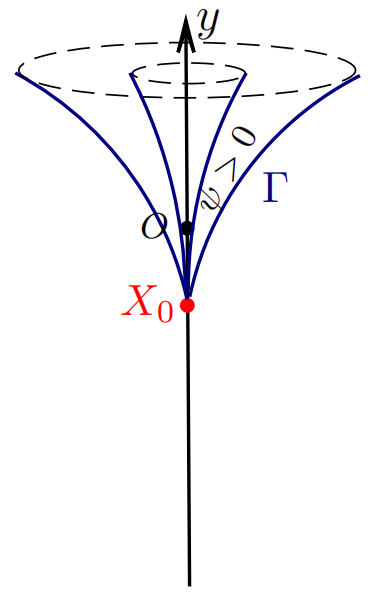}
	\end{minipage}%
	\begin{minipage}[t]{0.48\linewidth}
		\centering
		\includegraphics[width=3cm, height=5cm]{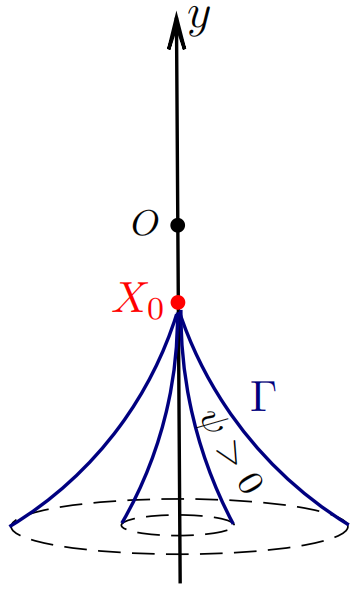}
	\end{minipage}%
		\caption*{Fig. 7. Cusp}

\end{figure}
Moreover, the blow-up limits $$\psi_{0}(x,y)=Cx^2\text{ or }\psi_{0}(x,y)\equiv0\text{ for Case 1}$$
 and
 $$\psi_{0}(x,y)\equiv0\text{ for Case 2}.$$
\end{theorem}

The third result shows the possible profiles of the free boundaries near the Type 3 degenerate point, i.e., the original point.

\begin{theorem}[Type 3 degenerate point]\label{theorem:3}
	  Let $\psi$ be a weak solution of the problem \eqref{equation} satisfying Assumption \ref{assumption1} and Assumption \ref{assumption:2}. Then for $X_{0}=O$,
	  \begin{equation*}
	  	\begin{aligned}
	  		\mathscr{D}^{xy}_{X_{0},\psi}(0+)\in\Bigg\{&-\int_{B_{1}^{+}\cap\{(r\sin\theta,r\cos\theta); 0<\theta<\theta^{*}\}}xy\,dX,\\&-\int_{B_{1}^{+}}xy^{+}\,dX,-\int_{B_{1}^{+}}xy^{-}\,dX,0\Bigg\}.
	  	\end{aligned}
	  \end{equation*}
	  There are exactly three cases.
	
	\textbf{Case 1. Garabedian pointed bubble.} If $\mathscr{D}^{xy}_{X_{0},\psi}(0+)=-\int_{B_{1}^{+}\cap\{(r\sin\theta,r\cos\theta); 0<\theta<\theta^{*}\}}xy\,dX$, then $g_{1}(t)\neq 0$ in $(0,t_{1})$, and \begin{equation*}
	\lim\limits_{t\rightarrow 0+}\frac{g_{2}(t)}{g_{1}(t)}=\tan\theta^{*} \text{ (see Fig. 8)},
	\end{equation*}
	where $\theta^{*}:=\arccos z^{*}$ and $z^{*}\in(-1,0)$ satisfies $P^{\prime}_{3/2}(z^{*})=0$. Here, $P_{3/2}(z)$ is the Legendre function of the first kind (see Remark \ref{remark:LE}).
\begin{figure}[H]
		\centering
		\includegraphics[width=6.4cm, height=4.8cm]{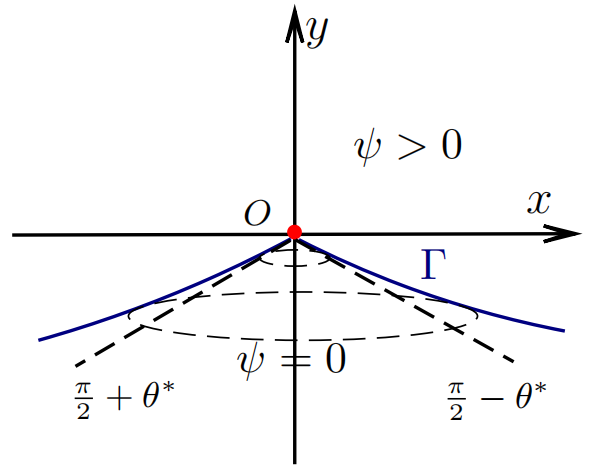}
	\caption*{Fig. 8. Garabedian pointed bubble}
\end{figure}
	
	\textbf{Case 2. Horizontally flatness.} If $\mathscr{D}^{xy}_{X_{0},\psi}(0+)\in\left\{-\int_{B_{1}^{+}}xy^{+}\,dX,-\int_{B_{1}^{+}}xy^{-}\,dX\right\}$, then $g_{1}(t)\neq0$ in $(0,t_{1})$ and \begin{equation*}
	\lim\limits_{t\rightarrow 0+}\frac{g_{2}(t)}{g_{1}(t)}=0 \text{ (see Fig. 9)}.
	\end{equation*}
\begin{figure}[H]
	\begin{minipage}[t]{0.48\linewidth}
		\centering
		\includegraphics[width=6.4cm, height=4.8cm]{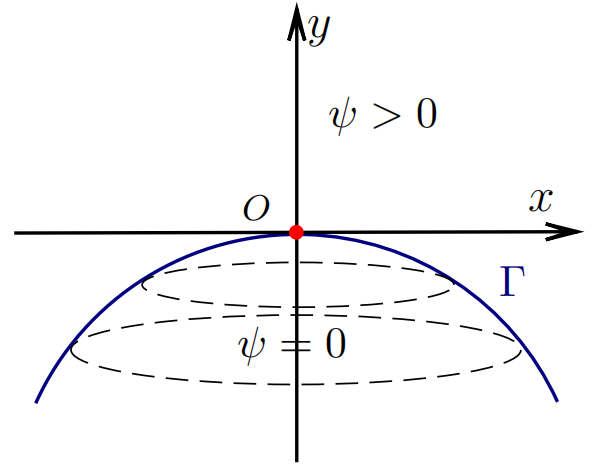}
\caption*{(a)}
	\end{minipage}%
	\begin{minipage}[t]{0.48\linewidth}
		\centering
		\includegraphics[width=6.4cm, height=4.8cm]{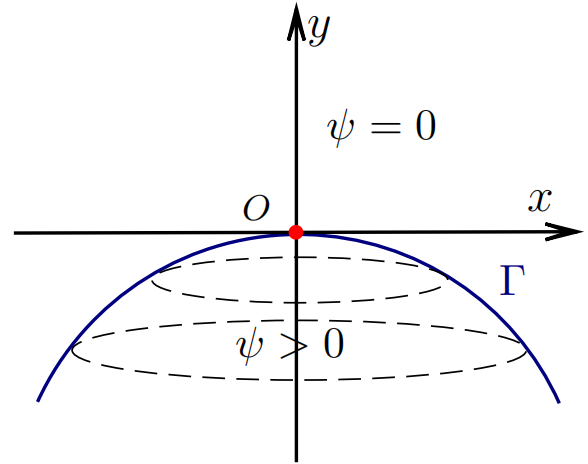}
\caption*{(b)}
	\end{minipage}%
\caption*{Fig. 9. Horizontally flatness}
\end{figure}
	
	\textbf{Case 3. Horizontally cusp.} If $\mathscr{D}^{xy}_{X_{0},\psi}(0+)=0$,  then $g_{1}(t)\neq0$ in $(-t_{1},t_{1})\backslash\{0\}$ and \begin{equation*}
	\lim\limits_{t\rightarrow 0}\frac{g_{2}(t)}{g_{1}(t)}=0 \text{ (see Fig. 10)}.
	\end{equation*}
\begin{figure}[H]
		\centering
		\includegraphics[width=6.4cm, height=4.8cm]{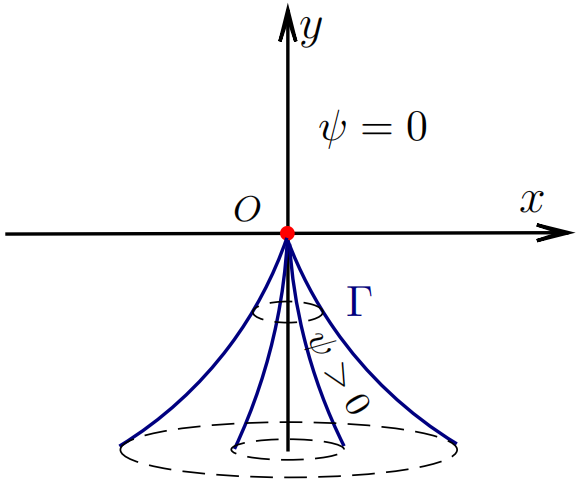}
	\caption*{Fig. 10. Horizontally cusp}
\end{figure}
In addition, the blow-up limits
$$\psi_{0}(x,y)=\psi_0(rsin\theta,rcos\theta)=C_0 r^{\frac{5}{2}}sin^2\theta P'_{3/2}(cos\theta)I_{\{(r sin\theta, r cos\theta):0<\theta<\theta^*\}}\text{ for Case 1 },$$
and $$\psi_{0}(x,y)\equiv0\text{ for Case 2 and Case 3}.$$
\end{theorem}

\begin{rmk}\label{remark:LE}
The form of a linear second-order differential equation
$$(1-x^2)\frac{d^2y}{dx^2}-2x\frac{dy}{dx}+a(a+1)y=0,$$
where $a$ is an arbitrary positive constant, is called the Legendre equation.

If $a$ is an integer, it follows from Section 7.3 in \cite{Lebedev} that the Legendre equation has two linearly independent solutions $P_{a}(z)$ and $Q_{a}(z)$ on $(-1,1)$, which are called the Legendre function of the first kind and the Legendre function of the second kind respectively.

If $a$ is not an integer, it follows from Section 7.2 in \cite{Lebedev} that the Legendre equation has two linearly independent solutions $P_{a}(z)$ and $P_{a}(-z)$ on $(-1,1)$.
\end{rmk}

\begin{rmk}
The  feature of Garabedian pointed bubble was first found by Garabedian in \cite{Garabedian1985},
he gave an example of an explicit solution of the problem \eqref{equation} for irrotational case, where the water dominated domain is a cone with vertex at the origin.
\end{rmk}

\begin{rmk}
In the next section, we will introduce two classical physical models on steady axisymmetric inviscid flows which formulate into the Bernoulli-type free boundary problem \eqref{equation}. The one is incompressible axis-symmetric rising jet in a channel, and the another one is the axis-symmetric bubble in a tube. In the first physical model, the air lies above the water, and in the second model, the water lies above the air. This is the main difference of the two physical models. In mathematical point of view, for the singularity on the axis, the up vertical cusp (Fig. 5 (a)) maybe formulate in the rising jet, and the down vertical cusp (Fig. 5 (b)) maybe occur in the second physical model.
\end{rmk}

The present paper is built up as follows. In Section 2, we first derive the Bernoulli's type free boundary problem \eqref{equation} from the two physical models on inviscid flow with general vorticity
in gravity field. Based on Weiss's monotonicity formula, we study the singularity near the  Type 1 degenerate point in Section 3. In Section 4 and Section 5, we investigate the singularity near the  Type 2 and Type 3 degenerate point, respectively.

			\section{Mathematical setting of the physical problem} In this section, we want to briefly introduce the derivation of the physical problems of steady axisymmetric flow with general vorticity in gravity field under some assumptions.

  The three-dimensional steady incompressible ideal water wave  is governed by the  steady Euler system in gravity field
  \begin{equation}\label{equation:eulerEquation1}
   \begin{cases}
    &\nabla\cdot\mathbf{u}=0,\\
    &(\mathbf{u}\cdot\nabla)\mathbf{u}+\nabla p+g\mathbf{e}^{3}=0,
   \end{cases}
  \end{equation}
 where $\mathbf{u}=(u_{1},u_{2},u_{3})$ is the velocity field, $p$   denotes the pressure, $g$ denotes the acceleration due to the gravity, $\mathbf{e}^{3}:=(0,0,1)$ and $(x_{1},x_{2},x_{3})\in\mathbb{R}^{3}$ is the space variable.
 And denote
  $\mathbf{w}=(w_{1},w_{2},w_{3})= \nabla\times\mathbf{u}$ as the vorticity of the fluid.

  Since we focus on the axisymmetric flows, we take $Y=x_{3}$ be the axis of symmetry and we let $x=\sqrt{x_{1}^{2}+x_{2}^{2}}$.  $u=\sqrt{u_{1}^{2}+u_{2}^{2}}$, $v_{\theta}$ and $v=u_{3}$  stand for the radial velocity, the swirl velocity  and the vertical velocity,   respectively. Moreover, let
  \begin{equation*}
   \mathbf{e}_{x}=\left(\frac{x_{1}}{x},\frac{x_{2}}{x},0\right), \mathbf{e}_{\theta}=\left(\frac{x_{2}}{x},-\frac{x_{1}}{x},0\right)\text{ and }\mathbf{e}_{Y}=(0,0,1),
  \end{equation*}
 be the standard orthonormal unit vectors in the cylindrical coordinate.
  Thus, the velocity field $\mathbf{v}$ in the axisymmetric coordinate can be written as
  \begin{equation*}
   \mathbf{v}=u\mathbf{e}_{x}+v_{\theta}\mathbf{e}_{\theta}+v\mathbf{e}_{Y}.
  \end{equation*}
  We only consider the flow without swirl in the present paper, namely,   $v_{\theta}\equiv 0$.

  Therefore, the Euler system \eqref{equation:eulerEquation1}  can be rewritten  as
  \begin{equation}\label{equation:eulerEquation2}
   \begin{cases}
    &(xu)_{x}+(xv)_{Y}=0,\\
    &(xu^{2})_{x}+(xuv)_{Y}+xp_{x}=0,\\
    &(xuv)_{x}+(xv^{2})_{Y}+xp_{Y}+g=0.
   \end{cases}
  \end{equation}
  And then, the vorticity of the fluid can be written as
  \begin{equation*}
   \begin{cases}
    &w_{1}=\frac{x_{2}}{x}\left(\frac{\partial v}{\partial x}-\frac{\partial u}{\partial Y}\right),\\
    &w_{2}=-\frac{x_{1}}{x}\left(\frac{\partial v}{\partial x}-\frac{\partial u}{\partial Y}\right),\\
    &w_{3}=0,
   \end{cases}
  \end{equation*}
 which gives that
  \begin{equation}\label{equation:rotational2}
   \mathbf{w}=\left(\frac{\partial v}{\partial x}-\frac{\partial u}{\partial Y}\right)\mathbf{e}_{\theta}\equiv w\mathbf{e}_{\theta}.
  \end{equation}
 Here, $w=\frac{\partial v}{\partial x}-\frac{\partial u}{\partial Y}$ is the scalar vorticity.

 In what follows, we would like to introduce two classical physical models in incompressible fluid.

 \subsection{Rising jet}\label{subsection:risingFlow}
  \begin{figure}[H]
  \includegraphics[width=100mm]{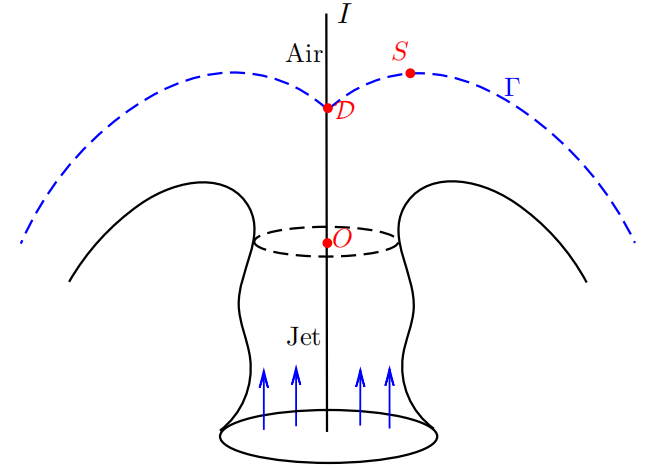}
  \caption{Axially symmetric rising jet issuing from a nozzle}\label{picture:model}
 \end{figure}
 The first physical problem reads as the incompressible inviscid fluid with vorticity issuing from a semi-infinitely long nozzle acted on by gravity, and the free surface initiates from some point $\mathcal{D}$ on the symmetric axis and extends to the far field (see Figure \ref{picture:model} and references \cite{Christodoulides2009} and \cite{Schulkes1994}).

  In order to clarify the physical problem, we give the definition of a semi-infinitely long nozzle as follows. As shown in Figure \ref{picture:mode2}, $N$ stands for the solid nozzle wall. Suppose that $N$ is a continuous injective curve, which can be written by $a(t)=(a_{1}(t),a_{2}(t))$ in $t\in[0,+\infty)$ such that $a(0)=(b_{1},-1)$, where $b_{1}$ is a positive constant.

 \begin{figure}[H]
  \includegraphics[width=60mm]{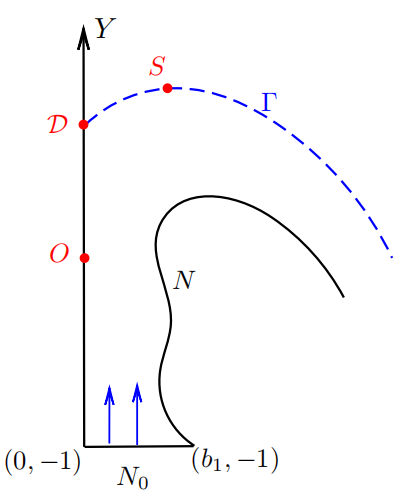}
  \caption{Axially symmetric rising jet issuing from a nozzle}\label{picture:mode2}
 \end{figure}
  On the other hand,  we assume that the flow issues from the entrance of the nozzle $N_0:=\{(x,-1)\mid 0\leq x\leq b_{1}\}$ and there exists  a free boundary $\Gamma$, which starts above the point $(0,-1)$ and does not intersect $N$. Moreover, denote by $I$ the symmetric axis, $D_{1}$ the possible fluid domain bounded by $I$,  $N$ and $N_0$, and $D_{2}$ the fluid domain bounded by $I$, $N$, $\Gamma$ and $N_0$.

  Next, we  want to introduce the vorticity-stream formulation for axisymmetric, inviscid flows without swirl.  The first equality in \eqref{equation:eulerEquation2} gives that there exists a stream function $\Psi$ such that
  \begin{equation*}
   u=\frac{1}{x}\frac{\partial\Psi}{\partial Y}\quad\text{ and } \quad v=-\frac{1}{x}\frac{\partial \Psi}{\partial x}.
  \end{equation*}
  We impose  that   the  vertical velocity and the scalar vorticity  in the inlet of the nozzle are $v_{0}(x)$ and $w_{0}(x)$ respectively. Moreover, we assume that $v_{0}(x)$ satisfies
 \begin{equation*}
  v_{0}(x)>0 \text{ for }x\in[0,b_{1}]\quad\text{ and }\quad \lim\limits_{x\rightarrow0+}\frac{w_{0}(x)}{x}\text{ exists},
 \end{equation*}
 and that the streamline is simple topological in the fluid domain $D_{2}$.

  Then, for any $(x,Y)\in D_{2}$, it can be pulled back along one streamline to the initial point $(\tilde{x},0)$ in the inlet. Thus we have
  \begin{equation*}
   \Psi=-2\pi\int_{0}^{\tilde{x}}sv_{0}(s)\,ds,
  \end{equation*}
  which together with the implicit function theorem gives that $\Psi$ is a function of $\tilde{x}$ and it is strictly increasing with respect to $\tilde{x}$. Therefore, we denote $\tilde{x}:=k(\Psi)$ as a function of $\Psi$.  In addition,  $\frac{w}{x}$ is an invariance along each streamline, that is,
  \begin{equation*}
   (u,v)\cdot\nabla\left(\frac{w}{x}\right)=0.
  \end{equation*}

  In view of the definition  \eqref{equation:rotational2}, a short calculation shows that the stream function $\Psi$ satisfies
  \begin{equation*}
   -\operatorname{div}\left(\frac{1}{x}\nabla\Psi\right)=w=x\cdot\frac{w}{x}=x\frac{w_{0}(k(\Psi))}{k(\Psi)}:=xh(\Psi).
  \end{equation*}

 Without loss of generality, we can impose the Dirichlet boundary  conditions as follows,
 \begin{equation*}
  \Psi=0\text{ on }I\cup\Gamma\quad\text{ and }\quad\Psi=-Q\text{ on }N,
 \end{equation*}
where $Q=\int_{0}^{b_{1}}2\pi xv_{0}(x)\,dx$.

  With the help of the equations of motion, we obtain the Bernoulli's law, which states that
  \begin{equation}\label{equation:bernoulli}
   \frac{1}{2}\left(u^{2}+v^{2}\right)+gY+p
  \end{equation}
  is a constant along each streamline. Moreover, we assume that the atmospheric pressure is a constant $p_{atm}$, and the constant pressure condition on the free surface gives that
  \begin{equation*}
   p=p_{atm}\quad\text{  on }\Gamma.
  \end{equation*}

  Thus, along the symmetric axis and the free boundary $\Gamma$, the identity \eqref{equation:bernoulli} gives that
  \begin{equation*}
  \frac{1}{2}v_{0}^{2}(0)+p(0,-1)=\frac{1}{2}\frac{\lvert\nabla\Psi\rvert^{2}}{x^{2}}+gY+p_{atm}\quad\text{ on }\Gamma.
  \end{equation*}
   We arrive at the boundary condition on $\Gamma$
  \begin{equation*}
   \begin{aligned}
    \frac{1}{2x^{2}}\lvert\nabla\Psi\rvert^{2}=\frac{1}{2}v_{0}^{2}(0)+p(0,-1)-p_{atm}-gY,
   \end{aligned}
  \end{equation*}
 which shows that
 \begin{equation*}
  \frac{1}{x^{2}}\lvert\nabla\tilde{\Psi}\rvert^{2}=\lambda-Y,
 \end{equation*}
where the  scaled function $\tilde{\Psi}$ is defined by $\tilde{\Psi}:=\frac{\Psi}{\sqrt{2g}}$ and $\lambda:=\frac{v_{0}^{2}(0)}{2g}+\frac{p(0,-1)}{g}-\frac{p_{atm}}{g}$. Moreover, we write $-y=\lambda-Y$. Then
  \begin{equation*}
   \frac{1}{x^{2}}\lvert\nabla\tilde{\Psi}\rvert^{2}=-y\text{  on }\Gamma.
  \end{equation*}
 And it is easy to check that the function $\tilde{\Psi}(x,Y)$ satisfies the following elliptic equation
 \begin{equation*}
  -\operatorname{div}\left(\frac{1}{x}\nabla\tilde{\Psi}\right)=x\frac{h(\sqrt{2g}\tilde{\Psi})}{\sqrt{2g}}.
 \end{equation*}

Therefore, taking $\psi(x,y)=-\tilde{\Psi}(x,\lambda+y)$ and $f(\psi)=\frac{h(\sqrt{2g}\tilde{\Psi})}{\sqrt{2g}}$, we formulate the following free boundary problem of the stream function that
  \begin{equation}\label{equation2}
   \begin{cases}
    \begin{alignedat}{2}
     \operatorname{div}\left(\frac{1}{x}\nabla \psi\right)&=-x f(\psi) \quad  &&\text{ in } D_{1} \cap\{\psi>0\},\\ \frac{1}{x^{2}}\left|\nabla \psi\right|^2&=-y  \quad  &&\text{ on }D_{1}\cap\partial\{\psi>0\},\\
     \psi=0\quad&\text{ on }I\cup\Gamma,\quad&&\psi=Q\quad\text{ on }N.
    \end{alignedat}
   \end{cases}
  \end{equation}

 \subsection{Bubble in a tube}
 The second physical model is a classical problem of a  bubble rising in a semi-finite long cylindrical tube (see references \cite{Doak2018} and \cite{Garabedian1985}). More precisely, it describes that an ideal rotational incompressible fluid acted on by gravity falls into a cylindrical nozzle and the free surface starts from some point $\mathcal{D}$ on the symmetric axis and goes into the far field, which is shown in Figure \ref{picture:fig1}.

 \begin{figure}[H]
  \includegraphics[width=50mm]{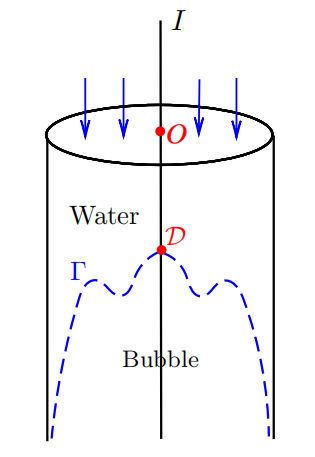}
  \caption{A  bubble rising in a semi-infinitely long cylindrical tube}\label{picture:fig1}
 \end{figure}
  Let $N$ be the nozzle wall and can be written by straight line: $\{(a,Y)\mid Y\leq0\}$ with $a>0$, $N_0:=\{(x,0)\mid 0\leq x\leq a\}$ is the entrance of the tube, see Figure \ref{picture:fig2}.
  \begin{figure}[H]
  \includegraphics[width=30mm]{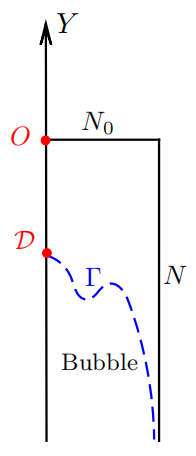}
  \caption{Bubble in a tube}\label{picture:fig2}
  \end{figure}
 Furthermore, we introduce a stream function $\Psi$ as \begin{equation*}
  u=\frac{1}{x}\frac{\partial\Psi}{\partial Y}\quad\text{ and }\quad v=-\frac{1}{x}\frac{\partial\Psi}{\partial x},
 \end{equation*}
and  impose that the vertical velocity and the scalar vorticity  in the inlet of the nozzle are $v_{0}(x)$ and $w_{0}(x)$, which satisfy
 \begin{equation*}
   v_{0}(x)<0\text{ for }x\in [0,a]\quad\text{ and }\quad\lim\limits_{x\rightarrow 0}\frac{w_{0}(x)}{x}\text{ exists.}
 \end{equation*}
Moreover, we assume that the streamlines are well-defined in the whole fluid field.

Following the notation as in Subsection \ref{subsection:risingFlow}, denote by $D_{1}$  the possible fluid domain bounded by $I$, $N$ and $N_{0}$, and $D_{2}$ the fluid domain bounded by $I$, $N$, $N_{0}$ and $\Gamma$.

Then similarly to the previous subsection, we formulate the following Bernoulli's type free boundary problem  of the scaled stream function $\psi(x,y)=-\frac{\Psi(x,\lambda+y)}{\sqrt{2g}}$
\begin{equation}\label{equation3}
 \begin{cases}
  \begin{alignedat}{2}
   \operatorname{div}\left(\frac{1}{x}\nabla \psi\right)&=-x f(\psi) \quad  &&\text{ in } D_{1} \cap\{\psi>0\},\\ \frac{1}{x^{2}}\left|\nabla \psi\right|^2&=-y  \quad  &&\text{ on }D_{1}\cap\partial\{\psi>0\},\\
   \psi=0&\text{ on }I\cup\Gamma,\quad&&\psi=Q\text{ on }N.
  \end{alignedat}
 \end{cases}
\end{equation}

The aim of this paper is to analyze the possible shapes of the free surface close to degenerate points. Thus we only consider the local behavior  of the solution near a degenerate point.

\section{The singularity near the Type 1 degenerate point}

In this section, we will investigate the possible wave profiles of the free boundary near the
 Type 1 degenerate point $X_{0}=(x_{0},0)$ for $x_{0}>0$. The main method we used is to analyze the possible blow-up limits in this point. In order to compute the possible explicit blow-up limits, we  use the monotonicity formula to compute the degree of the blow-up limit first.

Unless otherwise specified, we take $\psi$ be a variational solution of the problem \eqref{equation}, $X_{0}\in S_{\psi}^{s}$, and $r_{0}:=\min\{x_{0},\operatorname{dist}(X_{0},\partial\Omega)\}/2$ in this section.

The following lemma gains in interest if the first integral on the right-hand side of  \eqref{MFtwo} goes  to zero, then $\psi$ is a homogeneous function of degree  $\frac{3}{2}$.

\begin{lemma}\label{lemma:stagnationMonotonicity}
	   For almost everywhere $r\in(0,r_{0})$, we have
	\begin{equation}\label{MFtwo}
		\begin{aligned}
		\frac{d\mathscr{D}^{y}_{X_{0},\psi}(r)}{dr}	=&2r^{-3}\int_{\partial B_{r}(X_0)}\frac{1}{x}\left(\nabla \psi\cdot\nu-\frac{3}{2}\frac{\psi}{r}\right)^{2}\,d\mathcal{H}^1\\&-r^{-4}J_{1}(r)+\frac{3}{2}r^{-5}\int_{\partial B_{r}(X_0)}\frac{x-x_{0}}{x^{2}}\psi^{2}\,d\mathcal{H}^{1}-r^{-4}K_{1}(r),
		\end{aligned}
	\end{equation}
	where $\nu$ is the unit outer normal vector and
\begin{equation*}
	J_{1}(r)=\int_{ B_{r}(X_0)}\bigg(\frac{x-x_{0}}{x^{2}}\lvert\nabla\psi\rvert^{2}+(x-x_{0})yI_{\{\psi>0\}}\bigg)\,dX
\end{equation*}	
and \begin{equation*}
	\begin{aligned}
		K_{1}(r)=&\int_{ B_{r}(X_0)}(2F(\psi)(x-x_{0})+4xF(\psi))\,dX\\&-r\int_{ \partial B_{r}(X_0)}(2xF(\psi)-x\psi f(\psi))\,d\mathcal{H}^{1}.
	\end{aligned}
\end{equation*}
\end{lemma}

	\begin{proof}
	For almost everywhere $r\in(0,r_{0})$, recall that
	\begin{equation*}
		\mathscr{D}^{y}_{X_{0},\psi}(r)=r^{-3}\mathscr{D}_{1,X_{0},\psi}(r)-\frac{3}{2}r^{-4}\mathscr{D}_{2,X_{0},\psi}(r),
	\end{equation*}
where
\begin{equation*}
	\mathscr{D}_{1,X_{0},\psi}(r)=\int_{B_{r}(X_{0})}\left(\frac{1}{x}|\nabla \psi|^{2}-xyI_{\{\psi>0\}}-x\psi f(\psi)\right)\,dX,
\end{equation*}
and
\begin{equation*}
	\mathscr{D}_{2,X_{0},\psi}(r)=\int_{\partial B_{r}(X_{0})}\frac{1}{x}\psi^{2}\,d\mathcal{H}^{1}.
\end{equation*}

	First of all, we compute that
	\begin{equation}\label{equation4}
		\begin{aligned}
			\frac{d (r^{-3}\mathscr{D}_{1,X_{0},\psi}(r))}{dr}=&-3r^{-4}\int_{B_{r}(X_{0})}\left(\frac{1}{x}|\nabla \psi|^{2}-xyI_{\{\psi>0\}}-x\psi f(\psi)\right)\,dX\\&+r^{-4}\int_{\partial B_{r}(X_{0})}\left(\frac{1}{x}|\nabla \psi|^{2}-xyI_{\{\psi>0\}}-x\psi f(\psi)\right)\,d\mathcal{H}^{1}.
		\end{aligned}
	\end{equation}
	In order to get rid of the terms $3r^{-4}\int_{B_{r}(X_{0})}xyI_{\{\psi>0\}}\,dX$ and $-r^{-4}\int_{\partial B_{r}(X_{0})}xyI_{\{\psi>0\}}\,d\mathcal{H}^{1}$, we define the function $\zeta_{\varepsilon}(t)$ for   sufficiently small $\varepsilon>0$ as
	\begin{equation*}
		\zeta_{\varepsilon}(t):=\max\left(0,\min\left(1,\frac{r-t}{\varepsilon}\right)\right).
	\end{equation*}
	
	Taking $\eta_{\varepsilon}(X):=\zeta_{\varepsilon}(|X-X_{0}|)(X-X_{0})$ as a test function after approximation in \eqref{equation:variationalSolution} and letting $\varepsilon\rightarrow 0$, we  obtain that
	\begin{equation}\label{MFsteptwo}
		\begin{aligned}
			&\int_{ B_{r}(X_0)}3xyI_{\{\psi>0\}}\,dX-r\int_{\partial B_{r}(X_0)}xyI_{\{\psi>0\}}\,d\mathcal{H}^{1}\\
=&-r\int_{\partial B_{r}(X_{0})}\frac{|\nabla\psi|^{2}}{x}\,d\mathcal{H}^{1}-\int_{ B_{r}(X_0)}\frac{x-x_{0}}{x^{2}}|\nabla\psi|^{2}\,dX+r\int_{\partial B_{r}(X_{0})}\frac{2}{x}(\nabla\psi\cdot\nu)^{2}\,d\mathcal{H}^{1}\\
			&-\int_{ B_{r}(X_0)}(2F(\psi)(x-x_{0})+4xF(\psi))\,dX+r\int_{\partial B_{r}(X_0)}2xF(\psi)\,d\mathcal{H}^{1}\\
			&-\int_{ B_{r}(X_0)}y(x-x_{0})I_{\{\psi>0\}}\,dX.
		\end{aligned}
	\end{equation}

Plugging the left-hand side of \eqref{MFsteptwo} into \eqref{equation4}, we get
\begin{equation*}
	\begin{aligned}
		\frac{d (r^{-3}\mathscr{D}_{1,X_{0},\psi}(r))}{dr}=&-3r^{-4}\int_{ B_{r}(X_0)}\left(\frac{1}{x}|\nabla\psi|^{2}-x\psi f(\psi)\right)\,dX-r^{-4}K(r)-r^{-4}J_{1}(r)\\&+r\int_{ \partial B_{r}(X_0)}\frac{2}{x}(\nabla\psi\cdot\nu)^{2}\,d\mathcal{H}^{1}.
	\end{aligned}
\end{equation*}

	Next, noting that $\psi$ is a variational solution of the problem \eqref{equation}, an integration by parts shows  that for any $\delta>0$,
	\begin{align*}
	&~~~~\int_{ B_{r}(X_0)}\frac{1}{x}\nabla\psi\cdot\nabla(\max(\psi-\delta,0)^{1+\delta})\,dX\\
	&=-\int_{B_{r}(X_0)}\operatorname{div}\left(\frac{1}{x}\nabla\psi\right)\cdot(\max(\psi-\delta,0)^{1+\delta})\,dX\\
	&~~~~+\int_{\partial B_{r}(X_{0})}\frac{1}{x}\max(\psi-\delta,0)^{1+\delta}\nabla\psi\cdot\nu \,d\mathcal{H}^{1}\\
	&=\int_{B_{r}(X_0)}xf(\psi)(\max(\psi-\delta,0)^{1+\delta})\,dX\\
	&~~~~+\int_{\partial B_{r}(X_{0})}\frac{1}{x}\max(\psi-\delta,0)^{1+\delta}\nabla\psi\cdot\nu \,d\mathcal{H}^{1},
	\end{align*}
	which implies that
	\begin{equation}\label{MFstepthree}
	\int_{ B_{r}(X_0)}\frac{1}{x}|\nabla\psi|^{2}\,dX=\int_{ B_{r}(X_0)}x\psi f(\psi)\,dX+\int_{\partial B_{r}(X_0)}\frac{1}{x}\psi\nabla\psi\cdot\nu \,d\mathcal{H}^{1}
	\end{equation}
	as $\delta\rightarrow0$.
	
	 Combining \eqref{equation4} and \eqref{MFstepthree}, we deduce that
	\begin{equation}\label{MFstepfour}
			\begin{aligned}
			\frac{d (r^{-3}\mathscr{D}_{1,X_{0},\psi}(r))}{dr}
			=&-3r^{-4}\int_{\partial B_{r}(X_0)}\frac{1}{x}\psi\nabla\psi\cdot\nu \,d\mathcal{H}^{1}
			+r^{-3}\int_{\partial B_{r}(X_0)}\frac{2}{x}(\nabla\psi\cdot\nu)^{2}\,d\mathcal{H}^{1}\\
			&-r^{-4}J_{1}(r)-r^{-4}K_{1}(r).
		\end{aligned}
	\end{equation}

 Moreover, by  substitution of variables, a direct calculation shows that
	\begin{equation*}
			\begin{aligned}
			&\frac{d}{d r}\left(r^{-4} \int_{\partial B_{r}(X_{0})} \frac{1}{x} \psi^{2}\, d \mathcal{H}^{1}\right)\\=&\frac{d}{dr}\bigg(r^{-3}\int_{ \partial B_{1}}\frac{1}{x_{0}+rx}\psi^{2}(X_{0}+rX)\,d\mathcal{H}^{1}\bigg)\\=&-3 r^{-5} \int_{\partial B_{r}(X_{0})} \frac{1}{x} \psi^{2} \,d \mathcal{H}^{1}-r^{-5} \int_{\partial B_{r}(X_{0})}\frac{x-x_{0}}{x^{2}} \psi^{2} \,d \mathcal{H}^{1}\\&
			+2 r^{-4} \int_{\partial B_{r}(X_{0})} \frac{1}{x} \psi \nabla \psi \cdot \nu \,d \mathcal{H}^{1}
		\end{aligned}
	\end{equation*}
	for any $\psi\in W^{1,2}_{w,loc}(\Omega)$, which together with \eqref{MFstepfour} yields \eqref{MFtwo}. Thus we complete the proof of Lemma \ref{lemma:stagnationMonotonicity}.
\end{proof}

Here are some elementary properties of the  blow-up sequence $\{\psi_{m}\}$, the blow-up limit $\psi_{0}$ and weighted density  $\mathscr{D}_{X_{0},\psi}^{y}(r)$ as $r\rightarrow 0+$. Recall that $$\psi_{m}(X):=
\frac{\psi(X_{0}+r_{m}X)}{r_{m}^{3/2}}\quad \text{ for any }X_{0}\in S_{\psi}^{s}.$$
\begin{lemma}\label{lemma:stagnations}
	If $\psi$ is a variational solution of the problem \eqref{equation} satisfying Assumption 1.1, then the following statements hold.
	
	(\romannumeral1) The limit  $\mathscr{D}_{X_{0},\psi}^{y}(0+)=\lim\limits_{r\rightarrow 0+}\mathscr{D}_{X_{0},\psi}^{y}(r)$  exists. Moreover, the limit is finite.
	
	(\romannumeral2)  If $\psi_{m}(X)$ converges to $\psi_{0}$ weakly in $W_{loc}^{1,2}(\mathbb{R}^{2})$  as $r_{m}\rightarrow0+$, then  $\psi_{0}(r X)=r^{3/2}\psi_{0}(X)$  for each $X\in\mathbb{R}^{2}$ and $r>0$. Moreover, $\psi_{m}\rightarrow\psi_{0}$ strongly in $W_{loc}^{1,2}(\mathbb{R}^{2})$  as $r_{m}\rightarrow0+$.
	
	(\romannumeral3) The  weighted density $\mathscr{D}_{X_{0},\psi}^{y}(0+)$ satisfies
	\begin{equation*}
	\mathscr{D}_{X_{0},\psi}^{y}(0+)=-x_{0}\lim\limits_{m\rightarrow +\infty}\int_{B_{1}}y I_{\{\psi_{m}>0\}}\,dX.
	\end{equation*}
\end{lemma}

\begin{proof}
(\romannumeral1).
First, it follows from Assumption 1.1 and  the continuity of $f(\psi)$  that
 $r^{-4}J_{1}(r)$, $r^{-5}\int_{\partial B_{r}(X_{0})}\frac{x-x_{0}}{x^{2}}\psi^{2}\,d\mathcal{H}^{1}$ and $r^{-4}K_{1}(r)$  are integrable with respect to $r$ on  $(0,r_{0})$. Then we argue by contradiction. Assume that the limit $\lim\limits_{r\rightarrow 0+}\mathscr{D}_{X_{0},\psi}^{y}(r)$ does not exist.

 Thus there exist two sequences $\{r_{k}^{1}\}_{k=1}^{\infty}$ and $\{r_{k}^{2}\}_{k=1}^{\infty}$ with $\lim\limits_{k\rightarrow\infty}r_{k}^{i}=0$ for $i=1,2$ and $r_{k}^{1}<r_{k}^{2}$, such that
 \begin{equation*}
 	\lim\limits_{k\rightarrow +\infty}\mathscr{D}_{X_{0},\psi}^{y}(r_{k}^{1})=c+\delta\quad\text{ and }\quad\lim\limits_{k\rightarrow +\infty}\mathscr{D}_{X_{0},\psi}^{y}(r_{k}^{2})=c,
 \end{equation*}
where $c$ is a constant and $\delta$ is a positive constant.

On the other hand, we have
\begin{equation*}
	\begin{aligned}
		\mathscr{D}_{X_{0},\psi}^{y}(r_{k}^{2})-\mathscr{D}_{X_{0},\psi}^{y}(r_{k}^{1})&=\int_{r_{k}^{1}}^{r_{k}^{2}}\frac{d\mathscr{D}_{X_{0},\psi}^{y}(r)}{dr}\,dr\\
		&\geq\int_{r_{k}^{1}}^{r_{k}^{2}}\left(-r^{-4}J_{1}(r)+\frac{3}{2}r^{-5}\int_{\partial B_{r}(X_0)}\frac{x-x_{0}}{x^{2}}\psi^{2}\,d\mathcal{H}^{1}-r^{-4}K_{1}(r)\right)\,dr,
	\end{aligned}
\end{equation*}
which implies that
\begin{equation*}
	-\delta\geq0\text{ as }k\rightarrow\infty.
\end{equation*}
Hence we obtain a contradiction since $\delta>0$.

(\romannumeral2). In view of  Assumption \ref{assumption1}, it is easy to check that $\psi_{m}\in W^{1,\infty}(B_{r}(O))$, for each $0<r<+\infty$. Integrating the equality  \eqref{MFtwo} with respect to $r$ on $(r_{m}r_{1},r_{m}r_{2})$, where $0<r_{1}<r_{2}<r_{0}$ and $r_{m}\rightarrow0+$ as $m\rightarrow+\infty$, and changing variables, we obtain
\begin{align*}
2\int_{B_{r_{2}(O)}\backslash B_{r_{1}}(O)}&|X|^{-5}\frac{1}{x_{0}+r_{m}x}\left(\nabla \psi_{m}(X)\cdot X-\frac{3}{2}\psi_{m}(X)\right)^{2}\,dX\\&\leq \mathscr{D}_{X_{0},\psi}^{y}(r_{m}r_{2})-\mathscr{D}_{X_{0},\psi}^{y}(r_{m}r_{1})+\int_{r_{m}r_{1}}^{r_{m}r_{2}}r^{-4}J_{1}(r)\,dr\\&~~~~+\frac{3}{2}\int_{r_{m}r_{1}}^{r_{m}r_{2}}r^{-5}\int_{\partial B_{r}(X_0)}\frac{|x-x_{0}|}{x^{2}}\psi^{2}\,d\mathcal{H}^{1}dr+\int_{r_{m}r_{1}}^{r_{m}r_{2}}r^{-4}|K_{1}(r)|\,dr\\&\rightarrow 0\text{ as }m\rightarrow\infty.
\end{align*}

Since $\psi_{m}(X)$ converges to $\psi_{0}$ weakly in $W^{1,2}_{loc}(\mathbb{R}^{2})$, we arrive at the equality  $$\nabla \psi_{0}(X)\cdot X-\frac{3}{2}\psi_{0}(X)=0\text{ a.e. in }\mathbb{R}^{2}.$$     This gives that the blow-up limit $\psi_{0}$ is a homogeneous function of degree  $\frac{3}{2}$.

Next, we will show that $\psi_{m}\rightarrow\psi_{0}$ strongly in $W^{1,2}_{loc}(\mathbb{R}^{2})$.

It follows from the compact embedding $W^{1,\infty}(U)\subset\subset L^{\infty}(U)\subset\subset L^{2}(U)$, when $U$ is a bounded   domain, that there exists a subsequence still relabeled by $\{\psi_{m}\}$ such that $\psi_{m}\rightarrow\psi_{0}$ strongly in $L^{2}_{loc}(\mathbb{R}^{2})$.

Therefore, it suffices to show that there exists a subsequence still relabeled by $\{\psi_{m}\}$ such that $\nabla\psi_{m}\rightarrow\nabla\psi_{0}$ strongly in $L^{2}_{loc}(\mathbb{R}^{2})$.
With the help of  Proposition 3.32 in \cite{Brezis2010}, our aim is to prove that
\begin{equation}\label{equation:lem2.2}
	\limsup\limits_{m\rightarrow\infty}\parallel\nabla\psi_{m}\parallel_{L^{2}(U)}\leq\parallel\nabla\psi_{0}\parallel_{L^{2}(U)},
\end{equation} where $U$ is any compact subset of $\mathbb{R}^{2}$.

For each $m$, $\psi_{m}$ satisfies
\begin{equation*}
\operatorname{div}\left(\frac{1}{x_{0}+r_{m}x}\nabla\psi_{m}\right)=-(x_{0}+r_{m}x)r_{m}^{1/2}f(r_{m}^{3/2}\psi_{m})
\end{equation*}
in  $B_{r_{0}/r_{m}}(O)\cap\{\psi_{m}>0\}$.
Taking  $m\rightarrow\infty$, one has
\begin{equation}\label{equation:lem2.2.1}
	\Delta\psi_{0}=0\quad\text{ in }\{\psi_{0}>0\}
\end{equation}  since $\psi_{m}$ converges  locally uniformly to $\psi_{0}$ in $\{\psi_{m}>0\}$.

Therefore, as in the proof of the equation \eqref{MFstepthree} in Lemma \ref{lemma:stagnationMonotonicity}, we obtain
\begin{align*}
&\int_{\mathbb{R}^{2}}\frac{1}{x_{0}}|\nabla \psi_{m}|^{2} \eta\left(-\frac{r_{m}x}{x_{0}}+o\left(\frac{r_{m}x}{x_{0}}\right)^{2}\right) \,dX+\int_{\mathbb{R}^{2}} \frac{1}{x_{0}}|\nabla \psi_{m}|^{2} \eta \,dX\\
=&\int_{\mathbb{R}^{2}}  \frac{1}{x_{0}+r_{m}x}|\nabla\psi_{m}|^{2}\eta \,dX\\
=&-\int_{\mathbb{R}^{2}}\psi_{m}\operatorname{div}\left(\frac{1}{x_{0}+r_{m}x}\nabla\psi_{m}\right)\eta \,dX-\int_{\mathbb{R}^{2}} \psi_{m} \frac{1}{(x_{0}+r_{m} x)} \nabla \psi_{m} \cdot \nabla \eta \,dX,\\
\text{ which converges to }&-\int_{\mathbb{R}^{2}} \psi_{0}\Delta\psi_{0}\eta\,dX-\int_{\mathbb{R}^{2}} \psi_{0} \frac{1}{x_{0}} \nabla \psi_{0} \cdot \nabla \eta \,dX\\&=\frac{1}{x_{0}} \int_{\mathbb{R}^{2}}\left|\nabla \psi_{0}\right|^{2} \eta \,dX
\end{align*}
for any $\eta\in C^{1}_{0}(\mathbb{R}^{2})$ as $m\rightarrow\infty$,
 which implies \eqref{equation:lem2.2}.

(\romannumeral3). A direct calculation shows that
\begin{align*}
\mathscr{D}_{X_{0},\psi}^{y}(0+)=&\lim\limits_{r\rightarrow0+}r^{-3}\int_{B_{r}(X_{0})}\left(\frac{1}{x}|\nabla \psi|^{2}-xyI_{\{\psi>0\}}-x\psi f(\psi)\right)\,dX\\&-\lim\limits_{r\rightarrow0+}\frac{3}{2}r^{-4}\int_{\partial B_{r}(X_{0})}\frac{1}{x}\psi^{2}\,d\mathcal{H}^{1}\\
=&-\lim\limits_{m\rightarrow+\infty}\int_{B_{1}}(x_{0}+r_{m}x)(y_{0}+r_{m}y)I_{\{\psi_{m}>0\}}\,dX\\
&-\lim\limits_{m\rightarrow+\infty}\int_{B_{1}}(x_{0}+r_{m}x)r_{m}^{\frac{3}{2}}\psi_{m}f(r_{m}^{3/2}\psi_{m})\,dX\\
&+\lim\limits_{m\rightarrow+\infty}\int_{  B_{1}}\frac{1}{x_{0}+r_{m}x}\lvert\nabla\psi_{m}\rvert^{2}\,dX-\lim\limits_{m\rightarrow+\infty}\frac{3}{2}\int_{ \partial B_{1}}\frac{1}{x_{0}+r_{m}x}\psi_{m}^{2}\,d\mathcal{H}^{1}\\
=&-\lim\limits_{m\rightarrow+\infty}\int_{B_{1}}(x_{0}+r_{m}x)(y_{0}+r_{m}y)I_{\{\psi_{m}>0\}}\,dX\\
&+\int_{B_{1}}\frac{1}{x_{0}}\lvert\nabla\psi_{0}\rvert^{2}\,dX-\frac{3}{2}\int_{ \partial B_{1}}\frac{1}{x_{0}}\psi_{0}^{2}\,d\mathcal{H}^{1}.
\end{align*}

Recalling that $\psi_{0}(X)$ is a homogeneous function of degree $\frac{3}{2}$, we obtain that
\begin{equation*}
	\int_{B_{1}}\frac{1}{x_{0}}\lvert\nabla\psi_{0}\rvert^{2}\,dX=\frac{3}{2}\int_{ \partial B_{1}}\frac{1}{x_{0}}\psi_{0}^{2}\,d\mathcal{H}^{1},
\end{equation*}
which implies that
\begin{equation*}
	\mathscr{D}_{X_{0},\psi}^{y}(0+)=-x_{0}\lim\limits_{m\rightarrow +\infty}\int_{B_{1}}y I_{\{\psi_{m}>0\}}\,dX.
\end{equation*}
\end{proof}

As an immediate consequence of the preceding Lemma \ref{lemma:stagnationMonotonicity} and Lemma \ref{lemma:stagnations}, we have the following proposition, which follows closely to Theorem 3.8 in \cite{Varvaruca2012c}. We present the proof for the convenience of the reader. Moreover, note that if $\psi$ is a weak solution of \eqref{equation} satisfying Assumption 1.1, it is easy to check that $I_{\{\psi>0\}}$ is a function of bounded variation locally in $\Omega\cap\{x>0\}\cap\{y\neq 0\}$.

\begin{proposition}\label{pro:2.3}
Assume that  $\psi$ is a weak solution of the problem \eqref{equation} satisfying Assumption 1.1, then
 the possible blow-up limits and the corresponding weighted densities are:

either  $$\psi_{0}\left(r\sin\theta,r\cos\theta\right)=\frac{\sqrt{2}}{3}x_{0}r^{3/2}\cos\left(\frac{3}{2}\theta-\frac{\pi}{2}\right)I_{\left\{(r\sin\theta,r\cos\theta)\mid \frac{2}{3}\pi<\theta<\frac{4}{3}\pi\right\}}$$ and $$\mathscr{D}_{X_{0},\psi}^{y}(0+)=-x_{0}\int_{B_{1}\cap\left\{(r\sin\theta,r\cos\theta)\mid \frac{2}{3}\pi<\theta<\frac{4}{3}\pi\right\}}y \,dX;$$

or $$\psi_{0}\equiv0$$ and $$\mathscr{D}_{X_{0},\psi}^{y}(0+)\in\Bigg\{-x_{0}\int_{B_{1}}y^{+}\,dX,-x_{0}\int_{B_{1}}y^{-}\,dX,0\Bigg\}.$$
\end{proposition}

\begin{proof}
	Let $\{\psi_{m}\}$ be the blow-up sequence defined in \eqref{equation:blowUp1} and assume that $\psi_{m}$ converges to $\psi_{0}$ weakly in $W^{1,2}_{loc}(\mathbb{R}^{2})$ as $m\rightarrow+\infty$. Then thanks to Lemma \ref{lemma:stagnations} (\romannumeral2), $\psi_{m}$ converges to $\psi_{0}$ strongly in $W_{loc}^{1,2}(\mathbb{R}^{2})$. The rest proof is divided into three steps.
	
	\textbf{Step 1. } We would like  to derive the form of  the function $\psi_{0}(x,y)$ in the polar coordinates under the assumption that $\{\psi_{0}>0\}$ is not an empty set.
	
	Let us introduce polar coordinates $(r,\theta)$ with center at the origin, such that $\theta=0$ corresponds to the $y$-axis. Recalling the fact \eqref{equation:lem2.2.1}, we have
\begin{equation}\label{equation:pro1.3}
	\Delta\psi_{0}=0\text{ in }\{\psi_{0}>0\}.
\end{equation}

Note that $\psi_{0}$ is a homogeneous function of degree  $\frac{3}{2}$, thus $\psi_{0}(x,y)$ can be rewritten as $\psi_{0}(r,\theta)=r^{3/2}f(\theta)$.  It follows from \eqref{equation:pro1.3} that the function $f(\theta)$ satisfies
\begin{equation*}
		f''(\theta)+\frac{9}{4}f(\theta)=0\text{ in }\{\psi_{0}>0\}.
\end{equation*}
After a direct calculation, one obtains
\begin{equation*}
	f(\theta)=C_{1}\cos\left(\frac{3}{2}\theta+\theta_{1}\right).
\end{equation*}  Here $C_{1}$  is a positive constant and $\theta_{1}$   is a constant, which will be determined later.

Therefore, we see that each connected component of $\{\psi_{0}>0\}$ is a cone with vertex at the origin and an angle of 120 degrees. Since $\partial\{\psi>0\}$ is contained in the lower half-plane $\{(x,y)\mid x\geq 0,y\leq 0\}$, we have that $\partial\{\psi_{0}>0\}$ is contained in the lower half-plane $\{(x,y)\mid x\geq 0, y\leq 0\}$. Thus $\{\psi_{0}>0\}$  has at most one connected component.

\textbf{Step 2. } In order to determine the constants $C_{1}$ and $\theta_{1}$, we will verify the boundary condition on $\partial\{\psi_{0}>0\}$ under the assumption that $\{\psi_{0}>0\}$ is non-empty.

	For any $\eta=(\eta_{1},\eta_{2})\in C^{1}_{0}(\mathbb{R}^{2};\mathbb{R}^{2})$, $\eta_{m}(X)=(\eta_{m,1},\eta_{m,2}):=\eta\left(\frac{X-X_{0}}{r_{m}}\right)$. Then it follows from the equality \eqref{equation:variationalSolution} that
	\begin{equation}
		\begin{aligned}
			0=&\int_{\Omega}\bigg(\frac{1}{x}|\nabla\psi|^{2}\nabla\cdot\eta_{m}-\frac{1}{x^{2}}|\nabla \psi|^{2}\eta_{m,1}-\frac{2}{x}\nabla\psi \cdot D\eta_{m}\cdot\nabla\psi\notag\\
			&~~~~- \operatorname{div}(xy\eta_{m})I_{\{\psi>0\}}-2F(\psi)\operatorname{div}(x\eta_{m})\bigg)\,dX,
		\end{aligned}
	\end{equation}
	which implies that
\begin{equation}\label{equation:pro1.2}
	\begin{aligned}
		0=&\int_{\Omega_{0}}r_{m}^{2}\bigg(\frac{1}{x_{0}+r_{m}x}|\nabla\psi_{m}|^{2}\operatorname{div}\eta -r_{m}\frac{1}{(x_{0}+r_{m}x)^{2}}|\nabla \psi_{m}|^{2}\eta_{1}-\frac{2}{x_{0}+r_{m}x}\nabla\psi_{m} D\eta\cdot\nabla\psi_{m}\\
		&~~~~- (x_{0}+r_{m}x)yI_{\{\psi_{m}>0\}}\operatorname{div}\eta -(r_{m}y\eta_{1}+\left(x_{0}+r_{m}x\right)\eta_{2})I_{\{\psi_{m}>0\}}\\
		&~~~~-2F(r_{m}^{3/2}\psi_{m})\left(\eta_{1}+\frac{x_{0}+r_{m}x}{r_{m}}\operatorname{div}\eta\right)\bigg)\,dX,
	\end{aligned}
\end{equation}
where $\Omega_{0}:=\{X\in\mathbb{R}^{2}\mid X_{0}+r_{m}X\in\Omega\}$.

 Moreover,  we derive from the compact embedding from BV space into $L^{1}$ space that there exists a subsequence (not relabeled) $\{\psi_{m}\}$ and a function $I_{0}\in L^{1}_{loc}(\mathbb{R}^{2})$ such that $I_{\{\psi_{m}>0\}}$ converges to   $I_{0}$ strongly in $L^{1}_{loc}(\mathbb{R}^{2})$. Moreover,
 \begin{equation}\label{equation:3}
 	I_{0}\in\{0,1\}\text{ a.e. in }\mathbb{R}^{2}\quad\text{ and }\quad I_{0}=1\text{ in }\{\psi_{0}>0\}.
 \end{equation}

Taking $m\rightarrow+\infty$ in  \eqref{equation:pro1.2}, one has that $\psi_{0}$ is a homogeneous solution with
\begin{equation}\label{equation:pro1.5}
	\begin{aligned}
		0=\int_{\mathbb{R}^{2}}\left(\frac{1}{x_{0}}\left|\nabla\psi_{0}\right|^{2}\nabla\cdot\eta-\frac{2}{x_{0}}\nabla\psi_{0} D\eta\cdot\nabla\psi_{0}- x_{0}yI_{0}\operatorname{div}\eta -x_{0}\eta_{2}I_{0}\right)\,dX.
	\end{aligned}
\end{equation}

 We observe from
 \eqref{equation:pro1.5} that if $\psi_{0}=0$, then one has
\begin{equation*}
	\int_{\mathbb{R}^{2}}x_{0}I_{0}\operatorname{div}(y\eta)\,dX=0,
\end{equation*}
which gives that $I_{0}$ is a constant in each connected component of $\{\psi_{0}=0\}\cap\{y\neq 0\}$, denoted as $\bar{I}_{0}$.

Choosing any point $X_{0}\in\partial\{\psi_{0}>0\}\backslash\{O\}$, there exists some small $r_{0}>0$ such that  the outer normal to $\partial\{\psi_{0}>0\}$ has the constant normal $\nu(X_{0})$ in $B_{r_{0}}(X_{0})\cap\partial\{\psi_{0}>0\}$. Define $\eta(X):=\zeta(X)\nu(X_{0})$ for any $\zeta(X)\in C_{0}^{1}(B_{r_{0}}(X_{0}))$ and put $\eta(X)$ into the equation \eqref{equation:pro1.5}. Integrating by parts gives that
\begin{equation}\label{equation:pro1.7}
	\int_{B_{r_{0}}(X_{0})\cap\partial\{\psi_{0}>0\}}\lvert\nabla\psi_{0}\rvert^{2}\zeta\,d\mathcal{H}^{1}=\int_{B_{r_{0}}(X_{0})\cap\partial\{\psi_{0}>0\}}-x_{0}^{2}y(1-\bar{I}_{0})\zeta\,d\mathcal{H}^{1}.
\end{equation}
 Indeed, it follows from the Hopf's lemma that $\lvert\nabla\psi_{0}\rvert\neq0$ on $B_{r_{0}}(X_{0})\cap\partial\{\psi_{0}>0\}$. This yields that $\bar{I}_{0}\neq 1$. Thus one has  \begin{equation*}
 	\bar{I}_{0}=0.
 \end{equation*} Therefore,  \eqref{equation:pro1.7} in fact implies that
\begin{equation}\label{equation:pro1.8}
	\lvert\nabla\psi_{0}\rvert^{2}=-x_{0}^{2}y\quad\text{ on }\partial\{\psi_{0}>0\}.
\end{equation}

\textbf{ Step 3. } We will establish the possible explicit form of $\psi_{0}$ with the help of \eqref{equation:pro1.8}.

Since $\{\psi_{0}>0\}$ has at most one connected component, we have the following two possibilities.

\textbf{ Case 1.}  $\{\psi_{0}>0\}$ is empty, namely, $\psi_{0}\equiv0$ in $\mathbb{R}^{2}$. Note that \begin{equation*}
	\mathscr{D}_{X_{0},\psi}^{y}(0+)=-x_{0}\lim\limits_{m\rightarrow +\infty}\int_{B_{1}}y I_{\{\psi_{m}>0\}}\,dX.
\end{equation*}
We derive from \eqref{equation:pro1.5} that $\bar{I}_{0}$ is a constant either $0$ or $1$ in $\mathbb{R}^{2}$. Hence there are three possible weighted densities.

\textbf{Case 1.1.}
If $\bar{I}_{0}=0$ or $\bar{I}_{0}=1$ in $\{\psi_{0}=0\}\cap(\{y>0\}\cup\{y<0\})$, then $$\mathscr{D}_{X_{0},\psi}^{y}(0+)=0.$$

\textbf{Case 1.2.} If  $\bar{I}_{0}=0$ in $\{\psi_{0}=0\}\cap\{y>0\}$ and $\bar{I}_{0}=1$ in $\{\psi_{0}=0\}\cap\{y<0\}$, then $$\mathscr{D}_{X_{0},\psi}^{y}(0+)=-x_{0}\int_{B_{1}}y^{-}\,dX.$$

\textbf{Case 1.3.} If  $\bar{I}_{0}=0$ in $\{\psi_{0}=0\}\cap\{y<0\}$ and $\bar{I}_{0}=1$ in $\{\psi_{0}=0\}\cap\{y>0\}$, then $$\mathscr{D}_{X_{0},\psi}^{y}(0+)=-x_{0}\int_{B_{1}}y^{+}\,dX.$$

\textbf{ Case 2.} There exists one connected component of $\{\psi_{0}>0\}$. It follows from \eqref{equation:pro1.8} that
\begin{equation*}
	\frac{9}{4}C_{1}^{2}=-x_{0}^{2}\cos\theta,
\end{equation*}
which holds on both boundaries of the cone.

Let us define the angles of the two boundaries of the cone to $y$-axis are $\tilde{\theta}_{1}$ and $\tilde{\theta}_{2}$ respectively with $\tilde{\theta}_{1}<\tilde{\theta}_{2}$, then we have the following relationships
\begin{equation*}
	\begin{cases}
		\cos\tilde{\theta}_{1}&=\cos\tilde{\theta}_{2},\\
		\tilde{\theta}_{2}-	\tilde{\theta}_{2}&=\frac{2}{3}\pi.
	\end{cases}
\end{equation*}
Thus one has that $\tilde{\theta}_{1}=\frac{2}{3}\pi$ and $\tilde{\theta}_{2}=\frac{4}{3}\pi$, which yields that $C_{1}=\frac{\sqrt{2}}{3}x_{0}$. Then in the set $\{\psi_{0}>0\}$, the function $\psi_{0}$ can be written as
\begin{equation*} \psi_{0}(r\sin\theta,r\cos\theta)=\frac{\sqrt{2}}{3}x_{0}r^{3/2}\cos\left(\frac{3}{2}\theta+\theta_{1}\right),
\end{equation*}
where the parameter $\theta_{1}$ will be determined later.

Moreover, it follows from the boundary condition
\begin{equation*}
	\psi_{0}(r\sin\tilde{\theta}_{1},r\cos\tilde{\theta}_{1})=\frac{\sqrt{2}}{3}x_{0}r^{3/2}\cos(\pi+\theta_{1})=0
\end{equation*}
that $\theta_{1}=-\frac{\pi}{2}+k\pi$, where $k$ is an integer. Without loss of generality, we take $\theta_{1}=-\frac{\pi}{2}$. Therefore,
\begin{equation*}
	\psi_{0}(r\sin\theta,r\cos\theta)=\frac{\sqrt{2}}{3}x_{0}r^{3/2}\cos\left(\frac{3}{2}\theta-\frac{\pi}{2}\right)I_{\left\{(r\sin\theta,r\cos\theta)\mid \frac{2}{3}\pi<\theta<\frac{4}{3}\pi\right\}}.
\end{equation*}
The corresponding density is
\begin{equation*}
	\mathscr{D}_{X_{0},\psi}^{y}(0+)=-x_{0}\int_{B_{1}\cap\left\{(r\sin\theta,r\cos\theta)\mid \frac{2}{3}\pi<\theta<\frac{4}{3}\pi\right\}}y \,dX.
\end{equation*}
\end{proof}

With the aid of the preceding lemmas and proposition, we are now in a position to prove Theorem \ref{theorem:1}.

\begin{proof}[Proof of Theorem \ref{theorem:1}]
	
	 For each $(x,y)\in\mathbb{R}^{2}$, we define the polar coordinates $(x,y)=(r\sin\theta,r\cos\theta)$ with center at the origin, such that $\theta=0$ corresponds to $y$-axis.
	
	Since the free boundary $\partial\{\psi>0\}$ is contained in $\{(x,y)\mid x\geq0,y\leq 0\}$, we consider the sets
	\begin{multline*}
		A_{\pm}:=\bigg\{\theta_{0}\in\left[\frac{\pi}{2},\frac{3\pi}{2}\right]\bigg|\text{ there exists a sequence }t_{m}\rightarrow 0\pm\text{ as }m\rightarrow\infty
		 \text{ such that }\\
-\arg(g(t_{m})-g(0))+\frac{\pi}{2}\rightarrow\theta_{0} \bigg\},
	\end{multline*}
where $\arg X$ is defined as the complex argument of $X$.
	
 We claim that
 \begin{equation}\label{equation33}
 A_{\pm}\subset\left\{\frac{\pi}{2},\frac{2\pi}{3},\frac{4\pi}{3},\frac{3\pi}{2}\right\}.
 \end{equation}
	
	Suppose not, then there exists a sequence $t_{m}\rightarrow0$ as $m\rightarrow\infty$ such that $-\arg(g(t_{m})-g(0))+\frac{\pi}{2}$ converges to  $\theta_{0}$ and $\theta_{0}\in(A_{+}\cup A_{-})\backslash\left\{\frac{\pi}{2},\frac{2\pi}{3},\frac{4\pi}{3},\frac{3\pi}{2}\right\}$.
	
	Set $r_{m}:=|g(t_{m})-g(0)|$ and $\psi_{m}(X):=\frac{\psi(X_0+r_{m}X)}{r_{m}^{3/2}}$. And for each $\rho>0$, we take a ball
	 $B_{\rho}:=B_{\rho}(\sin\theta_{0},\cos\theta_{0})$ satisfying
	 \begin{equation*}
	 	B_{\rho}\bigcap\left(\{(x,0)\mid x\in\mathbb{R}\}\bigcup\left\{\left(x,-\frac{1}{\sqrt{3}}x+\frac{1}{\sqrt{3}}x_{0}\right)\mid x\in\mathbb{R}\right\}\bigcup\left\{\left(x,\frac{1}{\sqrt{3}}x-\frac{1}{\sqrt{3}}x_{0}\right)\mid x\in\mathbb{R}\right\}\right)=\emptyset.
	 \end{equation*}
	
	Since $\psi_{m}\rightarrow\psi_{0}$ strongly in $W^{1,2}_{loc}(\mathbb{R}^{2})$, we deduce that
	\begin{equation*}
		\operatorname{div}\left(\frac{1}{x_0+r_mx}\nabla\psi_{m}\right)(B_{\rho})\rightarrow\Delta\psi_{0}(B_{\rho})=0
	\end{equation*}
in the sense of measure.	
	On the other hand, since
	\begin{equation*}
		\operatorname{div}\left(\frac{1}{x}\nabla\psi_{m}\right)(B_{\rho})=\int_{B_{\rho}\cap\{\psi_{m}>0\}}-(x_{0}+r_{m}x)r_{m}^{\frac{1}{2}}f\left(r_{m}^{3/2}\psi_{m}\right)\,dX+\int_{B_{\rho}\cap\partial\{\psi_{m}>0\}}\sqrt{-y}\,d\mathcal{H}^{1}
	\end{equation*}
	and the curve length of  $B_{\rho}\cap\partial\{\psi_{m}>0\}$  is at least $2\rho-o(1)$, we also have
	\begin{equation*}
	\operatorname{div}\left(\frac{1}{x}\nabla\psi_{m}\right)(B_{\rho})\geq-cr_{m}^{1/2}+\int_{B_{\rho}\cap\partial\{\psi_{m}>0\}}\sqrt{-y}\,d\mathcal{H}^{1}\geq c(\theta_{0},\rho)-cr_{m}^{1/2},
	\end{equation*} which contradicts with the fact that $\operatorname{div}\left(\frac{1}{x}\nabla\psi_{m}\right)(B_{\rho})\rightarrow0$ as $m\rightarrow\infty$.
	
	 Recalling the Assumption \ref{assumption:2} that the free boundaries near the point $X_{0}$ is a continuous injective curve, we obtain  that for any $t\in I\backslash\{0\}$, $g_{1}(t)\neq 0$. Otherwise, we have $\pi\in A_{\pm}$, a contradiction.
	
	Moreover, it is easy to check that $A_{+}$ and $A_{-}$ are connected sets, which implies that $A_{+}$ and $A_{-}$ contain only one element. Therefore we define
the elements in $A_{+}$ and $A_{-}$	 as \begin{equation*}
		a_{+}:=-\lim\limits_{t_{m}\rightarrow0+}\arg(g(t_{m})-g(0))+\frac{\pi}{2}
	\end{equation*}  and
\begin{equation*}
	a_{-}:=-\lim\limits_{t_{m}\rightarrow0-}\arg(g(t_{m})-g(0))+\frac{\pi}{2}
\end{equation*}
respectively.
	
	In the following,	 we first consider the case  $\mathscr{D}_{X_{0},\psi}^{y}(0+)=-x_{0}\int_{B_{1}\cap\{(r\sin\theta,r\cos\theta)\mid\frac{2\pi}{3}<\theta<\frac{4\pi}{3}\}}y\,dX$.
	
	We take  $\hat{B}_{1}:=B_{\rho}\left(\frac{\sqrt{3}}{2},-\frac{1}{2}+\frac{1}{\sqrt{3}}x_{0}\right)$ and $\hat{B}_{2}:=B_{\rho}\left(-\frac{\sqrt{3}}{2},\frac{1}{2}-\frac{1}{\sqrt{3}}x_{0}\right), \rho>0$. It is easily seen that $\Delta\psi_{0}(\hat{B}_{i})>0, i=1, 2$, which implies that $\left\{a_{+},a_{-}\right\}=\left\{\frac{2}{3}\pi,\frac{4}{3}\pi\right\}$. Thus we complete the proof
in Case 1.
	
 Next, if $\mathscr{D}_{X_{0},\psi}^{y}(0+)\in\left\{-x_{0}\int_{B_{1}}y^{+}\,dX,-x_{0}\int_{B_{1}}y^{-}\,dX,0\right\}$, we know that $\Delta\psi_{0}(\hat{B}_{i})=0, i=1, 2$, which implies that $a_{+}$ and $a_{-}$ cannot belong to $\left\{\frac{2}{3}\pi,\frac{4}{3}\pi\right\}$; otherwise,  by using similar arguments of claim \eqref{equation33}, we could also derive a contradiction.
	
	Therefore, we obtain that $\left\{a_{+},a_{-}\right\}\subset\left\{\frac{1}{2}\pi,\frac{3}{2}\pi\right\}$. And $a_{+}\neq a_{-}$ implies that $$\mathscr{D}_{X_{0},\psi}^{y}(0+)\in\left\{-x_{0}\int_{B_{1}}y^{+}\,dX,x_{0}\int_{B_{1}}y^{-}\,dX,0\right\}$$ while $a_{+}= a_{-}$ implies that $\mathscr{D}_{X_{0},\psi}^{y}(0+)=0$. Thus we complete the proof in Case 2 and Case 3.
\end{proof}

\section{The singularity near the Type 2 degenerate point}

For this section we turn our attention to study the possible profiles of the free boundary near the Type 2 degenerate point. To this aim, we take $\psi$ be a variational solution of the problem \eqref{equation}   and $X_0\in S_{\psi}^{d}$ in this section unless otherwise stated.

We derive first the monotonicity formula at the  Type 2 degenerate point. Observe that if the first integral on the right-hand side of identity \eqref{MFthree}  equals to zero, then $\psi$ is a homogeneous function of degree $2$.

Let  $r_{0}:=\operatorname{dist}(X_{0},\partial\Omega)/2$ in this subsection.
\begin{lemma}\label{lemma:degenerateMonotonicity}
	For almost everywhere $r\in(0,r_{0})$, we have that
	\begin{equation}\label{MFthree}
		\begin{aligned}
			\frac{d\mathscr{D}_{X_{0},\psi}^{x}(r)}{dr}=&2 r^{-3} \int_{\partial B_{r}^{+}(X_{0})} \frac{1}{x}\left(\nabla \psi \cdot \nu-2 \frac{\psi}{r}\right)^{2} \,d \mathcal{H}^{1} \\
			&-r^{-4} \int_{B_{r}^{+}(X_{0})}(y-y_{0}) x I_{\{\psi>0\}} \,dX\\
			&-r^{-4}K_{2}(r)-r^{-4}\int_{ B_{r}^{+}(X_0)}x\psi f(\psi)\,dX,
		\end{aligned}
	\end{equation}
where \begin{equation*}
	\begin{aligned}
		K_{2}(r)=&\int_{ B_{r}^{+}(X_0)}6xF(\psi)\,dX-r\int_{ \partial B_{r}^{+}(X_0)}(2xF(\psi)-x\psi f(\psi))\,d\mathcal{H}^{1}.
	\end{aligned}
\end{equation*}
\end{lemma}

The proof follows by a simple modification of Lemma \ref{lemma:stagnationMonotonicity} and is therefore omitted here.

We shall show  some elementary properties of the blow-up sequence $\{\psi_{m}\}$ and the weighted density $\mathscr{D}_{X_{0},\psi}^{x}(r)$ as $r\rightarrow0+$ in the next lemma. Recall that $$\psi_{m}(X):=\frac{\psi(X_{0}+r_{m}X)}{r_{m}^{2}}\quad\text{ for any }X_{0}\in S_{\psi}^{a}.$$

\begin{lemma}\label{lemma:degenerate}	Suppose that the  Assumption \ref{assumption1} holds, then the following statements hold.
	
	(\romannumeral1) The limit $\mathscr{D}_{X_{0},\psi}^{x}(0+)=\lim\limits_{r\rightarrow 0+}\mathscr{D}_{X_{0},\psi}^{x}(r)$ exists and is finite.
	
	(\romannumeral2) If $\psi_{m}$ converges to $\psi_{0}$ weakly in $W_{w,loc}^{1,2}(\mathbb{R}^{2}_{+})$ as $m\rightarrow+\infty$, then $\psi_{0}(r X)=r^{2}\psi_{0}(X)$ for each $X\in\mathbb{R}^{2}_{+}$ and $r>0$. Moreover, $\psi_{m}\rightarrow\psi_{0}$ strongly in $W_{w,loc}^{1,2}(\mathbb{R}^{2}_{+})$ as $m\rightarrow+\infty$.
	
	(\romannumeral3) $\mathscr{D}_{X_{0},\psi}^{x}(0+)=-y_{0}\lim\limits_{m\rightarrow \infty}\int_{B_{1}^{+}}x I_{\{\psi_{m}>0\}}\,dX$.
\end{lemma}

\begin{proof}
	This results of (\romannumeral1) and (\romannumeral3) follow from (\romannumeral1) and (\romannumeral3) in Lemma \ref{lemma:stagnations} in the similar way, and the proof are therefore omitted.
	
	By the similar argument in  (\romannumeral2) in Lemma \ref{lemma:stagnations},
	we can prove  (\romannumeral2) in this lemma with some minor changes.
	
	In fact, we observe that $\operatorname{div}\left(\frac{1}{x}\nabla\psi_{m}\right)=-r_{m}^{2}xf(r_{m}^{2}\psi_{m})$ in $\{\psi_{m}>0\}$, which implies
	\begin{align*}
	\operatorname{div}\left(\frac{1}{x}\nabla\psi_{0}\right)=0\quad \text{ in }\{\psi_{0}>0\}.
	\end{align*}
	To show that $\limsup\limits_{m\rightarrow\infty}\parallel\nabla\psi_{m}\parallel_{L^{2}(U)}\leq\parallel\nabla\psi_{0}\parallel_{L^{2}(U)}$, where $U$ is any compact subset of $\mathbb{R}^{2}_{+}$,  we  compute that
	\begin{align*}
	\int_{\mathbb{R}_{+}^{2}} \frac{1}{x}|\nabla \psi_{m}|^{2} \eta \,dX&=-	\int_{\mathbb{R}_{+}^{2}}\psi_{m}\operatorname{div}\left(\frac{1}{x}\nabla\psi_{m}\eta\right) \,dX\\
	&=-\int_{\mathbb{R}_{+}^{2}}\psi_{m}\operatorname{div}\left(\frac{1}{x}\nabla\psi_{m}\right)\eta \,dX-\int_{\mathbb{R}_{+}^{2}} \psi_{m} \frac{1}{x} \nabla \psi_{m} \cdot \nabla \eta \,dX\\
	&=\int_{\mathbb{R}_{+}^{2}}\psi_{m}r_{m}^{2}xf(r_{m}^{2}\psi_{m})\eta \,dX-\int_{\mathbb{R}_{+}^{2}} \psi_{m} \frac{1}{x} \nabla \psi_{m} \cdot \nabla \eta \,dX\\
	&\rightarrow-\int_{\mathbb{R}_{+}^{2}} \psi_{0} \frac{1}{x} \nabla \psi_{0} \cdot \nabla \eta \,dX=\int_{\mathbb{R}_{+}^{2}} \frac{1}{x}|\nabla \psi_{0}|^{2} \eta \,dX,
	\end{align*}
	for any $\eta\in C^{1}_{0}(\mathbb{R}^{2}_{+})$.
\end{proof}

In the following, we will compute the possible explicit form of the blow-up limit $\psi_{0}(x,y)$. It should be noted that the main difficulty is that $\psi_{0}$ is not a harmonic function in $\{\psi_{0}>0\}$. Thus we want to introduce another function $\varphi_{0}$.

Notice that
\begin{equation*}
	\operatorname{div}\left(\frac{1}{x}\nabla \psi_{0}\right)=0\text{ in }\{\psi_{0}>0\}.
\end{equation*}
Hence, there exists a  function $\varPhi_{0}(x,y)$ in $\{\psi_{0}>0\}$ such that
\begin{equation*}
	\frac{\partial \varPhi_{0}}{\partial x}=-\frac{1}{x}\frac{\partial\psi_{0}}{\partial y}\quad\text{ and }\quad \frac{\partial \varPhi_{0}}{\partial y}=\frac{1}{x}\frac{\partial\psi_{0}}{\partial x}.
\end{equation*}
Let $\varphi_{0}(x_1,x_2,x_3)=\varPhi_{0}\left(\sqrt{x_1^2+x_2^2},x_3\right)$. A direct calculation shows that in $\{\psi_{0}>0\}$,
\begin{equation}\label{equation:pro3.3.1}
	\begin{aligned}
			\Delta_{(x,y,\gamma)}\varphi_{0}&=\frac{1}{x}\frac{\partial}{\partial x}\left(x\frac{\partial\varPhi_{0}}{\partial x}\right)+\frac{\partial^{2}\varPhi_{0}}{\partial y^{2}}\\
		&=\frac{1}{x}\frac{\partial}{\partial x}\left(-\frac{\partial\psi_{0}}{\partial y}\right)+\frac{\partial}{\partial y}\left(\frac{1}{x}\frac{\partial\psi_{0}}{\partial y}\right)\\
		&=0,
	\end{aligned}
\end{equation}
where $x=\sqrt{x_{1}^{2}+x_{2}^{2}},y=x_{3}$, $\gamma=\arctan\frac{x_{2}}{x_{1}}$, and we have used the fact that $\varphi_{0}$ is independent of $\gamma$. This implies that $\varphi_{0}(x_{1},x_{2},x_{3})$ is a harmonic function in $\{\psi_{0}>0\}$, which gives that
\begin{equation}\label{equation24}
	\Delta_{(r,\theta,\gamma)}\varphi_{0}=0\quad\text{ in }\{\psi_{0}>0\},
\end{equation}
where $r=\sqrt{x_{1}^{2}+x_{2}^{2}+x_{3}^{2}}$ and $\theta=\arctan\frac{\sqrt{x_{1}^{2}+x_{2}^{2}}}{x_{3}}$.

In addition, $\psi_{0}$ is a homogeneous function of degree $2$ yields that $\varphi_{0}$ is a homogeneous function of degree $1$. Combing with the fact that $\varphi_{0}$ is independent of $\gamma$,  we could rewrite  $\varphi_{0}(x_{1},x_{2},x_{3})$ as
\begin{equation}\label{equation23}
	\varphi_{0}(r,\theta)=rf(\cos\theta).
\end{equation}

Inserting \eqref{equation23} into \eqref{equation24}, we compute that
\begin{equation*}
	\begin{aligned}
		\Delta_{(r,\theta,\gamma)}\varphi_{0}&=\frac{1}{r^{2}}\frac{\partial}{\partial r}\left(r^{2}\frac{\partial\varphi_{0}}{\partial r}\right)+\frac{1}{r^{2}\sin\theta}\frac{\partial}{\partial\theta}\left(\sin\theta\frac{\partial\varphi_{0}}{\partial\theta}\right)+\frac{1}{r^{2}\sin^{2}\theta}\frac{\partial^{2}\varphi_{0}}{\partial\gamma^{2}}\\
		&=\frac{1}{r^{2}}\frac{\partial }{\partial r}(r^{2}f(\cos\theta))+\frac{1}{r^{2}\sin\theta}\frac{\partial}{\partial\theta}(r\sin\theta f'(\cos\theta)(-\sin\theta))\\&=0,
	\end{aligned}
\end{equation*}
from which we get that
\begin{equation}\label{equation:pro3.3.2}
	(1-z^{2})f''(z)-2zf'(z)+2f(z)=0\quad\text{ in }\{\psi_{0}>0\},
\end{equation}
where $z=\cos\theta$.

\begin{proposition}
	Assume that  $\psi$ is a weak solution satisfying Assumption 1.1, then the possible blow-up limits and the corresponding weighted densities are:
	
	either $$\psi_{0}(X)=Cx^{2}\quad\text{ with }C>0$$   and $$\mathscr{D}_{X_{0},\psi}^{x}(0+)=-y_{0}\int_{B_{1}^{+}}x\,dX;$$
	
	or $$\psi_{0}=0$$ and $$\mathscr{D}_{X_{0},\psi}^{x}(0+)\in\left\{0,-y_{0}\int_{B_{1}^{+}}x\,dX\right\}.$$
\end{proposition}

\begin{proof}
  It is shown that the function $\varphi_{0}$ satisfies the equation \eqref{equation:pro3.3.2} and
it follows from Section 7.3 in \cite{Lebedev} that the equation \eqref{equation:pro3.3.2} has two linearly independent solutions $P_{1}(z)$ and $Q_{1}(z)$ on $(-1,1)$, where
\begin{equation*}
	P_{1}(z)=z\quad\text{ and }\quad Q_{1}(z)=\frac{z}{2}\log\left(\frac{1+z}{1-z}\right)-1,\quad z\in(-1,1)
\end{equation*}  are called the Legendre function of the first kind and the Legendre function of the second kind respectively. Thus one has that the solutions of the equation \eqref{equation:pro3.3.2} are linear combinations of $P_{1}(z)$ and $Q_{1}(z)$ on $(-1,1)$, that is,
\begin{equation*}
	f(z)=aP_{1}(z)+bQ_{1}(z),
\end{equation*}
where some $a,b\in\mathbb{R}$.

Assume that $\{\psi_{0}>0\}$ is not-empty. Taking any connected component $V$ of $\{\psi_{0}>0\}$, let $\theta_{1}$ and $\theta_{2}$ with $\theta_{i}\in [0,\pi]$ for $i=1,2$ and  $\theta_{1}<\theta_{2}$ be such that
\begin{equation*}
	\begin{cases}
		\psi_{0}(r\sin\theta,r\cos\theta)>0&\text{ for }\theta\in(\theta_{1},\theta_{2}),\\
		\psi_{0}(r\sin\theta_{i},r\cos\theta_{i})=0&\text{ for }i=1,2.
	\end{cases}
\end{equation*}

Now, we claim that $b=0$. We consider two cases respectively.

\textbf{ Case 1. } $\theta_{1}=0$ or $\theta_{2}=\pi$. By virtue of the fact that
\begin{equation}\label{equation:pro3.3.3}
	\int_{V}x\lvert\nabla\varphi_{0}\rvert^{2}\,dX=\int_{V}\frac{1}{x}\lvert\nabla\psi_{0}\rvert^{2}\,dX<+\infty,
\end{equation}
one has that $b=0$. Otherwise, we check at once that
\begin{equation*}
	Q_{1}'(z)=\frac{1}{2}\log\left(\frac{1+z}{1-z}\right)+\frac{z}{(1+z)(1-z)},
\end{equation*}
which yields that
\begin{equation*}
	Q_{1}'(z)\rightarrow+\infty\quad\text{ as }z\rightarrow1\text{ or }z\rightarrow-1,
\end{equation*}
a contradiction to \eqref{equation:pro3.3.3}.

\textbf{ Case 2.} $0<\theta_{1}<\theta_{2}<\pi$. It is easily seen that $\nabla\varphi_{0}\cdot\nabla \psi_{0}=0$,
which implies that \begin{equation*}
	\nabla_{(\theta,r)}\psi_{0}\cdot\nabla_{(\theta,r)}\varphi_{0}=0\text{ when }r=1,
\end{equation*}
where $\nabla_{(r,\theta)}$ is the gradient in polar coordinates.
Since $\psi_0$ is a homogeneous function, one has
\begin{equation*}
	\lim\limits_{\theta^{*}\rightarrow\theta_{i}}\frac{\partial \varphi_{0}}{\partial\theta}\bigg|_{r=1,\theta=\theta^{*}}=	\lim\limits_{\theta^{*}\rightarrow\theta_{i}}\frac{1}{r\sin\theta}\frac{\partial \psi_{0}}{\partial r}\bigg|_{r=1,\theta=\theta^{*}}=0\quad\text{ for }i=1,2.
\end{equation*}
This gives that
\begin{equation*}
	f'(z_{i})=0,\quad\text{ where }z_{i}=\cos\theta_{i}\quad\text{ for }i=1,2.
\end{equation*}

Suppose that $b\neq 0$, then
\begin{equation*}
	\begin{cases}
		aP_{1}'(z_{1})+bQ_{1}'(z_{1})=0,\\
		aP_{1}'(z_{2})+bQ_{1}'(z_{2})=0,
	\end{cases}
\end{equation*}
which yields that
\begin{equation*}
	b(Q_{1}'(z_{1})-Q_{1}'(z_{2}))=0.
\end{equation*}
We get that $Q_{1}'(z_{1})= Q_{1}'(z_{2})$. However, $Q_{1}(z)$ is a strictly convex function since $Q_{1}''(z)=\frac{2}{(1-z^2)^2}>0$ when $z\in(-1,1)$. Thus it leads a contradiction.

Therefore,  we have shown the claim that $b=0$. In fact, in Case 2, we deduce that $a=b=0$. This yields that in any connected component $V$ of $\{\psi_{0}>0\}$,
\begin{equation*}
	\varphi_{0}(x,y)=ar\cos\theta=ay\quad\text{ and }\quad\text{ either }\theta_{1}=0\text{ or }\theta_{2}=\pi.
\end{equation*}
A direct calculation shows that in $\{\psi_{0}>0\}$,
\begin{equation*}
	\psi_{0}=ax^{2}\quad\text{ and  }\quad\theta_{1}=0,\theta_{2}=\pi,
\end{equation*}
where $a>0$.

Then preceding as in the proof of Proposition \ref{pro:2.3}, for any $\eta=(\eta_{1},\eta_{2})\in C_{0}^{1}(\mathbb{R}^{2},\mathbb{R}^{2})$ with $\eta_{1}=0$ on $\{x=0\}$, one has
\begin{equation}\label{equation:pro3.3.4}
	\begin{aligned}
		0=&\int_{\mathbb{R}^{2}_{+}}\left(\frac{1}{x}\left\lvert\nabla\psi_{0}\right\rvert^{2}\nabla\cdot\eta-\frac{2}{x}\nabla\psi_{0}D\eta\cdot\nabla\psi_{0}-\frac{1}{x^{2}}\left\lvert\nabla\psi_{0}\right\rvert^{2}\eta_{1}\right)\,dX\\&-y_{0}\int_{\mathbb{R}^{2}_{+}}(xI_{0}\operatorname{div}\eta+I_{0}\eta_{1})\,dX,
	\end{aligned}
\end{equation}
where $I_{0}$ is the strong $L_{loc}^{1}$-limit of $I_{\{\psi_{m}>0\}}$ along a subsequence. Similarly, $I_{0}=1$ in $\{\psi_{0}>0\}$.

Hence, we obtain the following two cases.

\textbf{ Case 1. } $\{\psi_{0}>0\}$ is not an empty set. Then $\psi_{0}(x,y)=ax^{2}$ with $a>0$. Recalling $(\romannumeral3)$ in Lemma \ref{lemma:degenerate}, the corresponding density is
\begin{equation*}
	\mathscr{D}_{X_{0},\psi}^{x}(0+)=-y_{0}\int_{B_{1}^{+}}x\,dX.
\end{equation*}

\textbf{ Case 2.} $\psi_{0}(x,y)\equiv 0$ in $\mathbb{R}^{2}_{+}$. Then the equation \eqref{equation:pro3.3.4} implies that
\begin{equation*}
	0=\int_{\mathbb{R}^{2}_{+}}(xI_{0}\operatorname{div}\eta+I_{0}\eta_{1})\,dX,
\end{equation*}
which yields that $I_{0}$ is a constant in $\mathbb{R}_{+}^{2}$ and its value is either $1$ or $0$.

Thus, the corresponding density has two possible cases, that is,
\begin{equation*}
	\mathscr{D}_{X_{0},\psi}^{x}(0+)\in \left\{-y_{0}\int_{B_{1}^{+}}x\,dX,0\right\}.
\end{equation*}
\end{proof}

Now we shall prove Theorem \ref{theorem:2}.
\begin{proof}
	We follow the notation used in the proof of Theorem \ref{theorem:1}.
	
	First, we claim that $A_{\pm}\subset\{0,\pi\}$. Suppose not, then we could find a sequence $t_{m}\rightarrow0$ as $m\rightarrow+\infty$, such that
	\begin{equation*}
		-\arg(g(t_{m})-g(t_{0}))+\frac{\pi}{2}\rightarrow\theta_{0}\quad\text{ and }\quad\theta_{0}\in A_{\pm}\backslash\{0\}.
	\end{equation*}

Similarly, set $r_{m}:=\lvert g(t_{m})-g(t_{0})\rvert$ and $\psi_{m}(X):=\frac{\psi(X_{0}+r_{m}X)}{r_{m}^{2}}$. For each $r_{0}>0$, choosing a ball $B_{r_{0}}:=B_{r_{0}}(\sin\theta_{0},\cos\theta_{0})$, such that
\begin{equation*}
	B_{r_{0}}\cap\{x=0\}=\emptyset,
\end{equation*}
one has
\begin{equation}\label{equation:thm1.1.1}
	\operatorname{div}\left(\frac{1}{x}\nabla\psi_{m}\right)\left(B_{r_{0}}\right)\rightarrow \operatorname{div}\left(\frac{1}{x}\nabla\psi_{0}\right)\left(B_{r_{0}}\right)=0,
\end{equation}
where we have used the fact that $\psi_{m}\rightarrow\psi_{0}$ strongly in $W^{1,2}_{w,loc}(\mathbb{R}^{2}_{+})$.

Besides, the facts that
\begin{equation*}
	\operatorname{div}\left(\frac{1}{x}\nabla\psi_{m}\right)(B_{r_{0}})=\int_{B_{r_{0}}^{+}(X_{0})\cap\{\psi_{m}>0\}}-r_{m}^{3/2}xf(r_{m}^{5/2}\psi_{m})\,dX+
	\int_{B_{r_{0}}^{+}(X_{0})\cap\partial\{\psi_{m}>0\}}\sqrt{-y}\,d\mathcal{H}^{1}
\end{equation*}
and that the curve length of  $B_{r_{0}}\cap\partial\{\psi_{m}>0\}$  is at least $2\rho-o(1)$ imply that
\begin{equation*}
	\operatorname{div}\left(\frac{1}{x}\nabla\psi_{m}\right)(B_{r_{0}})\geq-cr_{m}^{3/2}+\int_{B_{r_{0}}^{+}(X_{0})\cap\partial\{\psi_{m}>0\}}\sqrt{-y}\,d\mathcal{H}^{1}\geq c(\theta_{0},\rho)-cr_{m}^{3/2}>0.
\end{equation*}
Thus we obtain a contradiction to \eqref{equation:thm1.1.1}.

Next, we get that $A_{+}$ and $A_{-}$ are connected sets. And we define
\begin{equation*}
	a_{+}:=-\lim\limits_{t_{m}\rightarrow0+}\arg(g(t_{m})-g(t_{0}))+\frac{\pi}{2}\text{ and } a_{-}:=-\lim\limits_{t_{m}\rightarrow0-}\arg(g(t_{m})-g(t_{0}))+\frac{\pi}{2}.
\end{equation*}

Finally, by virtue of $A_{\pm}\subset\left\{0,\pi\right\}$, there are three possible cases.

\textbf{ Case 1.} Without loss of generality, there exists only $a_{+}$. Hence \begin{equation*}
	a_{+}=0\quad\text{ or }\quad a_{+}=\pi,
\end{equation*}
that is,
\begin{equation*}
	\lim\limits_{t\rightarrow0+}\frac{g_{1}(t)}{g_{2}(t)-y_{0}}=0,
\end{equation*}
and the corresponding density is $-y_{0}\int_{B_{1}^{+}}x\,dX$.

\textbf{ Case 2.} $a_{+}=0$ and $a_{-}=\pi$ (or conversely, $a_{+}=\pi$ and $a_{-}=0$). More precisely,
\begin{equation*}
	\lim\limits_{t\rightarrow0}\frac{g_{1}(t)}{g_{2}(t)-y_{0}}=0,
\end{equation*}
and the corresponding density is $-y_{0}\int_{B_{1}^{+}}x\,dX$.

\textbf{ Case 3.} $a_{+}=a_{-}=0$ or $a_{+}=a_{-}=\pi$. In this case, $g_{1}(t)-x_{0}$ does not change its sign at $t=0$ and the corresponding density is $0$.
\end{proof}

\section{The singularity near the Type 3 degenerate point}

In this section, we would like to  show the possible profile of the free boundary near the Type 3 degenerate point, namely, the original point. For this purpose, we take $\psi$ be a variational solution of the problem \eqref{equation}   and $X_0=(0,0)$ in this section unless otherwise stated.

Now, we derive the monotonicity formula at the original point. Note that if the first integral on the right-hand  side of \eqref{MFfour}  equals to zero, then $\psi$ is a homogeneous function of degree $\frac{5}{2}$.

Let  $r_{0}:=\operatorname{dist}(X_{0},\partial\Omega)/2$ in this section.
\begin{lemma}
	 For almost everywhere $r\in(0,r_{0})$, we have
	\begin{equation}\label{MFfour}
			\begin{aligned}
			\frac{d\mathscr{D}_{X_{0},\psi}^{xy}(r)}{dr}&=2 r^{-4} \int_{\partial B_{r}^{+}(X_{0})} \frac{1}{x}\left(\nabla \psi \cdot \nu-\frac{5}{2} \frac{\psi}{r}\right)^{2} d \mathcal{H}^{1}\\
			&-r^{-5}K_{2}(r)-r^{-5}\int_{ B_{r}^{+}(X_{0})}x\psi f(\psi)\,dX.
		\end{aligned}
	\end{equation}
\end{lemma}

The proof is similar to that of Lemma \ref{lemma:stagnationMonotonicity}, and so we omit it here.

We propose now to show the blow-up sequence $\left\{\psi_{m}\right\}$ and the weighted density $\mathscr{D}_{X_{0},\psi}^{xy}(r)$ as $r\rightarrow0+$. Recall that
$$\psi_{m}(X):=\frac{\psi(X_{0}+r_{m}X)}{r_{m}^{5/2}}\quad\text{ for }X_{0}=O.$$

\begin{lemma}\label{lemma:origin}
	If $\psi$ satisfies Assumption \ref{assumption1}, then,
	
	(\romannumeral1) The limit
	$\mathscr{D}_{X_{0},\psi}^{xy}(0+)=\lim\limits_{r\rightarrow 0+}\mathscr{D}_{X_{0},\psi}^{xy}(r)$ exists and is finite.
	
	(\romannumeral2) If $\psi_{m}(X)$ converges to $\psi_{0}$ weakly in $W_{w,loc}^{1,2}(\mathbb{R}^{2}_{+})$  as $m\rightarrow\infty$, then $\psi_{0}(r x)=r^{5/2}\psi_{0}(x)$ for each $x\in\mathbb{R}^{2}_{+}$ and $r>0$.  Moreover, $\psi_{m}\rightarrow\psi_{0}$ strongly in $W_{w,loc}^{1,2}(\mathbb{R}^{2}_{+})$ as $m\rightarrow\infty$.

	(\romannumeral3) $\mathscr{D}_{X_{0},\psi}^{xy}(0+)=\lim\limits_{m\rightarrow +\infty}-\int_{B_{1}^{+}}xy I_{\{\psi_{m}>0\}}\,dX.$
\end{lemma}

	\begin{proof}
	Observe that $\operatorname{div}\left(\frac{1}{x}\nabla\psi_{m}\right)=-r_{m}^{3/2}xf(r_{m}^{5/2}\psi_{m})$ in $\left\{\psi_{m}>0\right\}$, which yields that
	\begin{align*}
	\operatorname{div}\left(\frac{1}{x}\nabla\psi_{0}\right)=0\quad\text{ in }\left\{\psi_{0}>0\right\}.
	\end{align*}
	
	In addition, we calculate that
	\begin{align*}
	\int_{\mathbb{R}_{+}^{2}} \frac{1}{x}|\nabla \psi_{m}|^{2} \eta \,dX&=-	\int_{\mathbb{R}_{+}^{2}}\psi_{m}\operatorname{div}\left(\frac{1}{x}\nabla\psi_{m}\eta\right) \,dX\\
	&=-\int_{\mathbb{R}_{+}^{2}}\psi_{m}\operatorname{div}\left(\frac{1}{x}\nabla\psi_{m}\right)\eta \,dX-\int_{\mathbb{R}_{+}^{2}} \psi_{m} \frac{1}{x} \nabla \psi_{m} \cdot \nabla \eta \,dX\\
	&=\int_{\mathbb{R}_{+}^{2}}\psi_{m}r_{m}^{3/2}xf(r_{m}^{5/2}\psi_{m})\eta \,dX-\int_{\mathbb{R}_{+}^{2}} \psi_{m} \frac{1}{x} \nabla \psi_{m} \cdot \nabla \eta \,dX\\
	&\rightarrow-\int_{\mathbb{R}_{+}^{2}} \psi_{0} \frac{1}{x} \nabla \psi_{0} \cdot \nabla \eta \,dX=\int_{\mathbb{R}_{+}^{2}} \frac{1}{x}\left|\nabla \psi_{0}\right|^{2} \eta \,dX,
	\end{align*}
	for any $\eta\in C^{1}_{0}(\mathbb{R}^{2}_{+})$.
	
	Then proceeding as in the proof of Lemma \ref{lemma:stagnations}, we complete the proof the lemma, the details of which we omit.
\end{proof}

The possible explicit forms of $\psi_{0}$ will be shown in the next proposition.
\begin{proposition}
	Assume that  $\psi$ is a weak solution satisfying Assumption 1.1, then the possible blow-up limits and the corresponding weighted densities are:
	
	either $$\psi_{0}(r\sin\theta,r\cos\theta)=C_{0}r^{5/2}\sin^{2}\theta P_{3/2}^{\prime}(\cos\theta)I_{\{(r\sin\theta,r\cos\theta)\mid 0<\theta<\theta^{*}\}}$$ and $$\mathscr{D}_{X_{0},\psi}^{xy}(0+)=-\int_{B_{1}^{+}\cap\{(r\sin\theta,r\cos\theta)\mid 0<\theta<\theta^{*}\}}xy\,dX,$$
	where $C_{0}>0$ is a unique constant, $P_{3/2}(z)$ is the Legendre function of the first kind, and $\theta^{*}:=\arccos z^{*}$, where $z^{*}\in(-1,0)$ is the unique solution $z\in(-1,1)$ of $P^{\prime}_{3/2}(z)=0$;
	
	or $$\psi_{0}=0$$ and $$\mathscr{D}_{X_{0},\psi}^{xy}(0+)\in\left\{0,-\int_{B_{1}^{+}}xy^{+}\,dX,-\int_{B_{1}^{+}}xy^{-}\,dX\right\}.$$
\end{proposition}

\begin{proof}
	Using the fact that $\operatorname{div}\left(\frac{1}{x}\nabla\psi_{0}\right)=0$ in $\{\psi_{0}>0\}$ and proceeding as  in Section 4, we can define a function $\varphi_{0}(x_{1},x_{2},x_{3})$ in $\{\psi_{0}>0\}$ as $$\varphi_{0}(x_{1},x_{2},x_{3})=\varPhi_{0}\bigg(\sqrt{x_{1}^{2}+x_{2}^{2}},x_{3}\bigg),$$
	where $\varPhi_{0}(x,y)$ is defined by
	$$\frac{\partial \varPhi_{0}}{\partial x}=-\frac{1}{x}\frac{\partial\psi_{0}}{\partial y}\quad\text{ and }\quad \frac{\partial \varPhi_{0}}{\partial y}=\frac{1}{x}\frac{\partial\psi_{0}}{\partial x}.$$ Clearly, $\varphi_{0}$ is a harmonic function in $\left\{\psi_{0}>0\right\}$ and is a homogeneous function of degree $\frac{3}{2}$.

Then similar to the derivation of the equation \eqref{equation:pro3.3.2}, we could rewrite $\varphi_{0}(x_{1},x_{2},x_{3})$ as \begin{equation*}
	\varphi_{0}(r,\theta)=r^{3/2}f(\cos\theta),
\end{equation*}
and the function $f$ satisfies the Legendre differential equation
\begin{equation}\label{equation:pro4.3.1}
(1-z^{2})f''(z)-2zf'(z)+\frac{15}{4}f(z)=0\quad\text{ in }\{\psi_{0}>0\},
\end{equation}
where $z=\cos\theta$.

It follows from Section 7.2 in \cite{Lebedev} that $P_{3/2}(z)$ and $P_{3/2}(-z)$ are two linear independent solutions of the equation \eqref{equation:pro4.3.1}. Thus the function $f$ can be written as
\begin{equation*}
	f(z)=aP_{3/2}(z)+bP_{3/2}(-z)\quad\text{ in }\{\psi_{0}>0\},\quad z\in (-1,1),
\end{equation*}
where $a,b\in\mathbb{R}$. Moreover, the following facts are well-known and are helpful for our proof
\begin{equation}\label{equation:pro4.3.2}
	\begin{aligned}
		&(1)\, z\longmapsto P_{3/2}(z) \text{ is real-analytic in a neighborhood of }z=1,\\
		&(2)\, P_{3/2}(z)\sim\frac{1}{\pi}\log\left(\frac{1+z}{2}\right),P_{3/2}'(z)\sim-\frac{1}{\pi}\frac{1}{1+z}\text{ as }z\rightarrow-1,\\
		&(3)\, P_{3/2}'(z)\text{ has a unique root }z^{*}\in(-1,0) \text{ in }(-1,1),\\
		&(4)\,P_{3/2}'(z)<0\text{ for }z\in(-1,z^{*})\text{ and }P_{3/2}'(z)>0\text{ for }z\in(z^{*},1).
	\end{aligned}
\end{equation}

In the following, we assume that the set $\{\psi_{0}>0\}$ is not-empty. Considering again any connected component $V$ of $\{\psi_{0}>0\}$, let $\theta_{1},\theta_{2}\in[0,\pi]$ with $\theta_{1}<\theta_{2}$ be such that
\begin{equation*}
	\begin{cases}
		\psi_{0}(r\sin\theta,r\cos\theta)>0&\text{ for }\theta\in(\theta_{1},\theta_{2}),\\
		\psi_{0}(r\sin\theta_{i},r\cos\theta_{i})=0&\text{ for }i=1,2.
	\end{cases}
\end{equation*}
Besides, by using similar methods in Proposition \ref{pro:2.3}, one has that
\begin{equation}\label{equation:pro4.3.3}
	\int_{V}x\lvert\nabla\varphi_{0}\rvert^{2}\,dX=\int_{V}\frac{1}{x}\lvert\nabla\psi_{0}\rvert^{2}\,dX<+\infty.
\end{equation}
and  if $0<\theta_{1}<\theta_{2}<\pi$, then
\begin{equation}\label{equation:pro4.3.4}
	f'(z_{i})=0,\text{ where }z_{i}=\cos\theta_{i}\quad\text{ for }i=1,2.
\end{equation}

We claim that the case $\theta_{1}=0$ and $\theta_{2}=\pi$ is impossible. In fact, if $\theta_{1}=0$, then $z_{1}=1$. It follows from the fact (2) in \eqref{equation:pro4.3.2} that
\begin{equation*}
	P_{3/2}'(-z)\sim\frac{1}{\pi}\frac{1}{1-z}\quad\text{ as }z\rightarrow1,
\end{equation*}
which implies that $b=0$. Otherwise, we obtain a contradiction to \eqref{equation:pro4.3.3}. Similarly, if $\theta_{2}=\pi$, we have that $a=0$. Thus the proof of the claim is completed.

We now have three possible cases.

\textbf{ Case 1.} $\theta_{1}=0$, then $b=0$.  One has that
\begin{equation*}
	f(z)=aP_{3/2}(z)\quad\text{ for }z\in(z^{*},1).
\end{equation*}
On the other hand, by virtue of the fact that $\theta_{2}\neq \pi$, \eqref{equation:pro4.3.4} implies that
\begin{equation*}
	f'(z_{2})=0,
\end{equation*}
which, together with (3) and (4) in \eqref{equation:pro4.3.2}, yields that
\begin{equation*}
	z_{2}=z^{*}:=\cos\theta^{*},
\end{equation*}
that is,
\begin{equation*}
	\theta_{1}=0,\quad\theta_{2}=\theta^{*}\text{ and }f(z)=aP_{3/2}(z)\text{ for }z\in(z^{*},1).
\end{equation*}

\textbf{ Case 2.}   $\theta_{2}=\pi$, then $a=0$. Along the proof of Case 1, one has that
\begin{equation}
\theta_{1}=\pi-\theta^{*},\quad\theta_{2}=\pi\text{ and }	f(z)=bP_{3/2}(-z)\quad\text{ for }s\in(-1,-z^{*}).
\end{equation}

\textbf{ Case 3.} $0<\theta_{1}<\theta_{2}<\pi$, then $f'(z_{1})=f'(z_{2})=0$.

First, we claim that
\begin{equation*}
	a\neq 0\text{ and }b\neq 0.
\end{equation*}
Indeed, suppose that $a= 0$ and $b\neq 0$, then \begin{equation*}
	P_{3/2}'(-z_{1})=P_{3/2}'(-z_{2})=0,
\end{equation*}
a contradiction to the fact (3) in \eqref{equation:pro4.3.2}. Similar considerations apply to the case when $a\neq0$ and $b=0$.

Next, we would like to show that
\begin{equation*}
	z_{1}\neq z^{*}\text{ and }z_{2}\neq z^{*}.
\end{equation*}
Suppose to the contrary  that $z_{1}=z^{*}$, then \eqref{equation:pro4.3.4} gives that $-bP_{3/2}'(-z_{1})=0$, which implies that $-z_{1}=z^{*}$.

Thus we obtain that
\begin{equation*}
	z_{1}=-z_{1}=z^{*}=0.
\end{equation*}
However, this contradicts with the fact that $z^{*}\in (-1,0)$. If $z_{2}=z^{*}$, by using the same argument, we could lead also a contradiction.

In the following, we want to estimate the range of $\theta_{1}$ and $\theta_{2}$. Observing that
\begin{equation*}
	\frac{P_{3/2}'(-z_{1})}{P_{3/2}'(z_{1})}=\frac{P_{3/2}'(-z_{2})}{P_{3/2}'(z_{2})}=\frac{a}{b},
\end{equation*}
we introduce the function $w(z):(-1,1)\backslash\left\{z^{*}\right\}\rightarrow\mathbb{R}$ defined by
\begin{equation*}
	w(z)=\frac{P_{3/2}'(-z)}{P_{3/2}'(z)}.
\end{equation*}
It follows from \eqref{equation:pro4.3.2} that
\begin{equation}\label{equation:pro4.3.5}
	\begin{aligned}
		\lim\limits_{z\rightarrow-1}w(z)=0,	\lim\limits_{z\rightarrow z^{*}-}w(z)=+\infty,	\lim\limits_{z\rightarrow z^{*}+}w(z)=-\infty\text{ and }	\lim\limits_{z\rightarrow1}w(z)=+\infty.
	\end{aligned}
\end{equation}

On the other hand, we compute that
\begin{equation}\label{equation:pro4.3.7}
	\begin{aligned}
		w'(z)&=\frac{P_{3/2}''(-z)P_{3/2}'(z)-P_{3/2}''(z)P_{3/2}'(-z)}{(P_{3/2}'(z))^{2}}\\
		&=\frac{15}{4}\frac{1}{1-z^{2}}\frac{P_{3/2}(z)P_{3/2}'(-z)-P_{3/2}(-z)P_{3/2}'(-z)}{(P_{3/2}'(z))^{2}}\\
		&=\frac{15}{2\pi}\frac{1}{(1-z^{2})^{2}}\frac{1}{(P_{3/2}'(z))^{2}}>0,
	\end{aligned}
\end{equation}
where we have used the facts that $P_{3/2}(z)$ and $P_{3/2}(-z)$ are solutions of the Legendre differential equation \eqref{equation:pro4.3.1} and that
\begin{equation*}
	P_{3/2}(z)P_{3/2}'(-z)-P_{3/2}(-z)P_{3/2}'(-z)=-\frac{2\sin\left(\frac{3}{2}\pi\right)}{\pi}\frac{1}{1-z^{2}}
\end{equation*}
( see Section 7.7 in \cite{Lebedev}).

Besides, we derive from \eqref{equation:pro4.3.5} that
\begin{equation*}
	\lim\limits_{z\rightarrow -z^{*}-}w(z)=	\lim\limits_{z\rightarrow -z^{*}+}w(z)=0,
\end{equation*}
which implies that
\begin{equation}\label{equation:pro4.3.6}
	w\left(-z^{*}\right)=0.
\end{equation}

Combing with \eqref{equation:pro4.3.5}, \eqref{equation:pro4.3.7} and \eqref{equation:pro4.3.6}, one has that
\begin{equation*}
	z_{1}\in(-z^{*},1)\quad\text{ and }\quad z_{2}\in(-1,z^{*}),
\end{equation*}
that is,
\begin{equation*}
	\theta_{1}\in (0,\pi-\theta^{*})\subset\left(0,\frac{\pi}{2}\right)\text{ and }	\theta_{2}\in (\theta^{*},\pi)\subset\left(\frac{\pi}{2},\pi\right).
\end{equation*}

Furthermore, since $w(z_{1})=w(z_{2})$,  the non-zero constants $a$ and $b$ can be determined uniquely by the equations $f'(z_{1})=f'(z_{2})=0$.

However, since $\partial\{\psi>0\}$ is contained in the lower half-plane, one has that
\begin{equation*}
	\operatorname{div}\left(\frac{1}{x}\nabla\psi_{0}\right)=0\quad\text{ in }\{y>0\},
\end{equation*}
which implies that Case 2 and Case 3 are impossible.

Therefore, we obtain that for any connected component $V$ of $\{\psi_{0}>0\}$, $\theta_{1}=0$ and $\theta_{2}=\theta^{*}$. Moreover, $f(z)=aP_{3/2}(z)$ for $z\in(z^{*},1)$ with $a\neq 0$. Meanwhile,
\begin{equation*}
	\varphi_{0}(r\sin\theta,r\cos\theta)=ar^{3/2}P_{3/2}(\cos\theta)I_{\{(r\sin\theta,r\cos\theta)\mid0<\theta<\theta^{*}\}},
\end{equation*}
which implies that
\begin{equation}\label{equation:pro4.3.8}
	\psi_{0}(r\sin\theta,r\cos\theta)=\frac{2}{5}ar^{5/2}\sin^{2}\theta P_{3/2}'(\cos\theta)I_{\{(r\sin\theta,r\cos\theta)\mid0<\theta<\theta^{*}\}}.
\end{equation}

In what follows, preceding as in the proof of Proposition \ref{pro:2.3}, for any
$\eta=(\eta_{1},\eta_{2})\in C_{0}^{1}(\mathbb{R}^{2},\mathbb{R}^{2})$ with $\eta_{1}=0$ on $\{x=0\}$, one has that
\begin{equation}\label{equation:pro4.3.9}
	\begin{aligned}
		0=&\int_{\mathbb{R}^{2}_{+}}\left(\frac{1}{x}\left\lvert\nabla\psi_{0}\right\rvert^{2}\nabla\cdot\eta-\frac{2}{x}\nabla\psi_{0}D\eta\cdot\nabla\psi_{0}-\frac{1}{x^{2}}\lvert\nabla\psi_{0}\rvert^{2}\eta_{1}\right)\,dX\\&-\int_{\mathbb{R}^{2}_{+}}(xyI_{0}\operatorname{div}\eta+yI_{0}\eta_{1}+xI_{0}\eta_{2})\,dX,
	\end{aligned}
\end{equation}
where $I_{0}$ is the strong $L_{loc}^{1}$-limit of $I_{\{\psi_{m}>0\}}$ along a subsequence.

Now, we consider the following cases.

\textbf{ Case 1.} There exists one connected component of $\{\psi_{0}>0\}$. Then $\psi_{0}$ is given by \eqref{equation:pro4.3.8}.

For any point $X_{0}=(r\sin\theta^{*},r\cos\theta^{*})\in\partial\{\psi_{0}>0\}$, there exist some small $r_{0}>0$ such that the outer normal to $\partial\{\psi_{0}>0\}$ has the constant normal $\nu(X_{0})$ in $B_{r_{0}}(X_{0})$. Taking $\eta(X):=\zeta(X)\nu(X_{0})$ for any $\zeta\in C_{0}^{1}(B_{r_{0}}(X_{0}))$ and putting $\eta(X)$ into \eqref{equation:pro4.3.9}, one has that
\begin{equation*}
	\int_{B_{r_{0}}(X_{0})\cap\partial\{\psi_{0}>0\}}\frac{1}{x}\lvert\nabla\psi_{0}\rvert^{2}\zeta\,d\mathcal{H}^{1}=\int_{B_{r_{0}}(X_{0})\cap\partial\{\psi_{0}>0\}}xy(1-\bar{I}_{0})\zeta\,d\mathcal{H}^{1},
\end{equation*}
where $\bar{I}_{0}$ denotes the constant of $I_{0}$ in $\{\psi_{0}=0\}$.

It follows from \eqref{equation:pro4.3.9} that $\bar{I}_{0}\in\{0,1\}$. On the other hand, the Hopf's lemma gives that $\lvert\nabla\psi_{0}\rvert\neq 0$ on $B_{r_{0}}(X_{0})\cap\partial\{\psi_{0}>0\}$, which implies that $\bar{I}_{0}=0$. Hence, we obtain that
\begin{equation*}
	\lvert\nabla\psi_{0}\rvert^{2}=x^{2}y^{2}\quad\text{ on }\partial\{\psi_{0}>0\}\cap\{x>0\},
\end{equation*}
which determine uniquely $a$ in \eqref{equation:pro4.3.8}.

And the corresponding density is
\begin{equation*}
	\mathscr{D}_{X_{0},\psi}^{xy}(0+)=-\int_{B_{1}^{+}\cap\{(r\sin\theta,r\cos\theta)\mid 0<\theta<\theta^{*}\}}xy\,dX.
\end{equation*}

\textbf{ Case 2.} $\psi_{0}(x,y)\equiv0$. \eqref{equation:pro4.3.9} gives that
$\bar{I}_{0}$ is a constant in each connected component of $\{\psi_{0}=0\}\cap\{y\neq 0\}$, either $1$ or $0$. Therefore, by using the similar arguments in the Proposition in \ref{pro:2.3}, the corresponding densities are
\begin{equation*}
	\mathscr{D}_{X_{0},\psi}^{xy}(0+)\in\left\{0,-\int_{B_{1}^{+}}xy^{+}\,dX,-\int_{B_{1}^{+}}xy^{-}\,dX\right\}.
\end{equation*}
\end{proof}

Now, as in the proof of Theorem \ref{theorem:1}, we can derive Theorem \ref{theorem:3}.

\begin{proof}
	We follow the same notations in Theorem \ref{theorem:1}. Here however we claim that
	\begin{equation*}
		A_{\pm}\subset\left\{\frac{\pi}{2},\theta^{*},\pi\right\}.
	\end{equation*}
	
	We choose the ball $B_{\rho}$ to satisfy
	\begin{equation*}
		B_{\rho}\cap(\{x=0\}\cup\left\{y=0\right\}\cup\{(r\sin\theta^{*},r\cos\theta^{*})\mid r\in\mathbb{R}_{+}\})=\emptyset.
	\end{equation*}
	Noting that
	\begin{equation*}
		\operatorname{div}\left(\frac{1}{x}\nabla\psi_{m}\right)\left(B_{\rho}\right)\rightarrow \operatorname{div}\left(\frac{1}{x}\nabla\psi_{0}\right)\left(B_{\rho}\right)=0,
	\end{equation*}
and  that the curve length of  $B_{\rho}\cap\partial\{\psi_{m}>0\}$  is at least $2\rho-o(1)$, we lead a contradiction that
\begin{equation*}
	\operatorname{div}\left(\frac{1}{x}\nabla\psi_{m}\right)(B_{r_{0}})>0.
\end{equation*}

Meanwhile, it is clear that $A_{+}$ and $A_{-}$ are connected sets. Thus we define that
\begin{equation*}
	a_{+}:=-\lim\limits_{t_{m}\rightarrow0+}\arg(g(t_{m})-g(t_{0}))+\frac{\pi}{2}\text{ and } a_{-}:=-\lim\limits_{t_{m}\rightarrow0-}\arg(g(t_{m})-g(t_{0}))+\frac{\pi}{2}.
\end{equation*}

Therefore, there are three possible cases.

\textbf{ Case 1.}  If $\psi_{0}(r\sin\theta,r\cos\theta)=C_{0}r^{5/2}\sin^{2}\theta P_{3/2}^{\prime}(\cos\theta)I_{\{(r\sin\theta,r\cos\theta)\mid 0<\theta<\theta^{*}\}}$, then $a_{+}=\theta^{*}$. In fact, we choose a ball $B_{1}:=B_{\frac{1}{10}}(r\sin\theta^{*},r\sin\theta^{*})$. One has that
   \begin{equation*}
	\operatorname{div}\left(\frac{1}{x}\nabla\psi_{m}\right)\left(B_{1}\right)>0,
\end{equation*}
which  implies that
\begin{equation*}
	A_{+}=\left\{\theta^{*}\right\}.
\end{equation*}

\textbf{ Case 2.} If $\psi_{0}\equiv0$  and $\mathscr{D}^{xy}_{X_{0},\psi}(0+)\in\left\{-\int_{B_{1}^{+}}xy^{+}\,dX,-\int_{B_{1}^{+}}xy^{-}\,dX\right\}$, then $a_{+}=\frac{\pi}{2}$.

\textbf{ Case 3.} If $\psi_{0}\equiv0$  and $\mathscr{D}^{xy}_{X_{0},\psi}(0+)=0$, then  $a_{+}=a_{-}=\pi$.
\end{proof}

\end{document}